\documentclass[11pt]{article}
\usepackage{subfig}
\usepackage{amsmath}
\usepackage{amssymb}
\usepackage{graphicx}
\usepackage{dcolumn}
\usepackage{mathtools}
\usepackage{amsmath,amsfonts}
\usepackage{bm}
\usepackage{amsmath}
\usepackage{amssymb}
\usepackage{algorithm2e}
\usepackage{xcolor}
\usepackage{algorithmic}
\usepackage{float}
\usepackage{setspace}
\usepackage{tabularx}
\usepackage{cite}
\usepackage{amsthm}

\allowdisplaybreaks

\newtheorem{thm}{Theorem}[section]
\newtheorem{prop}[thm]{Proposition}
\newtheorem{lem}[thm]{Lemma}

\newtheorem{ex}[thm]{Example}

\newtheorem{ass}[thm]{Assumption}

\newcommand{\f}{\frac}
\newcommand{\ds}{\displaystyle}
\newcommand{\E}{ {\mathbb{E}} }

\newcommand{\be}{\begin{equation}}
\newcommand{\ee}{\end{equation}}

\def\ds{\displaystyle}

     \def\({\Big (}       
      \def\){\Big )}       
      \def\[{\Big[}        
          \def\]{\Big]}        
              
               \def\1n{\negthinspace}
           \def\2n{\1n\1n}      
         \def\3n{\1n\2n}

\def\ba{\begin{array}}                \def\ea{\end{array}}
\def\bel{\begin{equation}\label}      \def\ee{\end{equation}}

\makeatletter
   
   \@addtoreset{equation}{section}
\makeatother
  \allowdisplaybreaks[4]
\def\be{\begin{equation}}
\def\bel{\begin{equation}\label}
\def\ee{\end{equation}}
\def\bea{\begin{eqnarray}}
\def\eea{\end{eqnarray}}
\def\bt{\begin{theorem}\label}
\def\et{\end{theorem}}
\def\bc{\begin{corollary}\label}
\def\ec{\end{corollary}}
\def\bex{\begin{example}\label}
\def\ex{\end{example}}
\def\bl{\begin{lemma}\label}
\def\el{\end{lemma}}
\def\bp{\begin{proposition}\label}
\def\ep{\end{proposition}}
\def\br{\begin{remark}\label}
\def\er{\end{remark}}
\def\ba{\begin{array}}
\def\ea{\end{array}}
\def\bde{\begin{definition}\label}
\def\ede{\end{definition}}

\begin{document}
\title{Convergence Analysis for Training Stochastic Neural Networks via Stochastic Gradient Descent}
\author{
Richard Archibald\thanks{ Computational Science and Mathematics Division, Oak Ridge National Laboratory, Oak Ridge, Tennessee  };
\and Feng Bao\thanks{ Department of Mathematics, Florida State University, Tallahassee, Florida, \ ({\tt bao@math.fsu.edu})};
\and Yanzhao Cao \thanks{ Department of Mathematics, Auburn University, Auburn, Alabama,  } 
\and Hui Sun \thanks{ Department of Mathematics, Florida State University, Tallahassee, Florida.}
}

\maketitle
\begin{abstract}
In this paper, we carry out numerical analysis to prove convergence of a novel sample-wise back-propagation method for training a class of stochastic neural networks (SNNs). The structure of the SNN is formulated as discretization of a stochastic differential equation (SDE). A stochastic optimal control framework is introduced to model the training procedure, and a sample-wise approximation scheme for the adjoint backward SDE is applied to improve the efficiency of the stochastic optimal control solver, which is equivalent to the back-propagation for training the SNN. The convergence analysis is derived with and without convexity assumption for optimization of the SNN parameters. Especially, our analysis indicates that the number of SNN training steps should be proportional to the square of the number of layers in the convex optimization case. Numerical experiments are carried out to validate the analysis results, and the performance of the sample-wise back-propagation method for training SNNs is examined by benchmark machine learning examples.

\end{abstract} 



\section{Introduction}

Deep neural network (DNN) based machine learning techniques are positioned to fundamentally change many sectors of society by offering decision making capabilities, which match and often exceed that of human experts \cite{ML_future_17, DL_Nature_15, silver2017mastering, EfficientPrO2017Sze, td-gammon1994, AlphaGo2017Silver2017,agent57-2020}. Despite of dramatic success that DNNs achieved, a closer examination reveals inherent challenges of applying such approaches broadly to science and engineering. A major challenge is that DNN models are sensitive to noises in data. A recent study \cite{Error_ML} shows that effectiveness of different DNN models are affected by inaccurate data in the datasets, and science and engineering solutions typically need to be able to quantify the uncertainty of DNN outputs and to provide effective operating ranges and risks. 

One  successful effort  to address the challenge of uncertainty quantification for DNN is probabilistic learning, which aims to incorporate randomness to DNN models and produce random output. Then, one can use the statistical behaviors of DNN's random output to study the uncertainty of data. The state-of-the-art probabilistic learning approach is the Bayesian neural network (BNN) \cite{Geneva2019, Geneva2020, Kwon2020, McDermott2019, Savchenko2020, Wu2020, Yang2021,Yao2019}, which treats parameters in a DNN as random variables and approximates their distributions through Bayesian inference. Instead of searching for the optimal parameters by deterministic optimization, the BNN utilizes Bayesian optimization to derive parameter distributions. As a result, the estimated random parameters generate random output, which characterizes the target model's uncertainty. Although the BNN approach provides a principled conceptual framework to quantify the uncertainty in probabilistic machine learning, carrying out Bayesian optimization to estimate a massive number of parameters in the BNN often renders the exact inference of parameter posteriors intractable. Moreover, the large size of the parameter set leads to uninterpretable parameter priors.  
Some of the existing Non-Bayesian methods, such as \cite{Gal2016, Lak2017}, are simple to implement but either ad hoc or computationally prohibitive.  


In this work, we consider a type of stochastic neural networks that models uncertainty in some of the well known DNN structures (e.g., residual neural networks, convolutional neural networks) through discretized stochastic (ordinary) differential equations (SDEs). Recently, an ordinary differential equation (ODE) interpretation for DNNs (named Neural-ODE) has been introduced and studied \cite{Chen2018, Dupont2019, E2017, Gerstberger1997, Haber2018}. The central idea of the Neural-ODE is to formulate the evolution of potentially huge hidden layers in the DNN as a discretized ODE system. To characterize the randomness caused by the uncertainty of models and noises of data,  we add an additive Brownian motion noise to the ODE to count for the model uncertainty of the DNN.  This changes the ODE to an SDE, and the deterministic DNN becomes a stochastic neural network (SNN) \cite{ Kong2020, Liu2020, Tzen2019}.  In the SNN model, the drift parameters serve as the prediction of the network, and the stochastic diffusion governs the randomness of network output, which serves to quantify the uncertainty of deep learning.  In comparison with the BNN, which needs to calibrate tremendous amount of random parameters through Bayesian inference, the SNN only uses one noise term (i.e., the diffusion term) to control the amount of uncertainty at each layer. Apparently, the computational cost for evaluating the diffusion coefficient is much lower than high-dimensional Bayesian inference for a large number of unknown random parameters. On the other hand, for an SNN with multiple hidden layers, by stacking (controlled) diffusion terms together through the multi-layer structure, we would still be able to characterize sufficient probabilistic behavior of the neural network.

The bottleneck of the SNN is constructing an efficient numerical solver for the back-propagation. This process propagates the total loss back into the neural network and updates the parameters accordingly to minimize the loss. It is the essence of neural network training. For Neural-ODE, back-propagation is equivalent to solving the adjoint of the ODE, which is a backward ODE. Since its structure is no different from that of a forward ODE, back-propagation can be achieved by means of a deterministic adjoint ODE solver \cite{Chen2018}.  For SNNs, the back-propagation is equivalent to solving the adjoint of the SDE, which is a backward SDE. Because of the martingale nature of the backward SDE \cite{BSDE_finance}, one must create a separate backward SDE solver which is completely different from and far more computationally intensive than forward SDE solvers. Thus the back-propagation by completely solving the backward SDE has been deemed unfeasible \cite{Tzen2019}. Several  alternatives have been proposed \cite{Kong2020, Liu2020, Tzen2019}. But they either fail to consider the Ito's nature of stochastic differentiation or lack scalability. Thus these training frameworks have been deemed inefficient \cite{Li2021}. 
In a recent study \cite{Bao_SNN}, we formulated the training procedure for SNN as a stochastic optimal control problem, and introduced a sample-wise scheme for solving the backward SDE corresponding to the forward propagation of the SNN and constructed a novel back-propagation solver.  A central  idea of our method is to treat the random samples of the backward SDE as ``pseudo data", and only solve the backward SDE partially on randomly selected pseudo data. In this way, only a tiny fraction of the computing cost of solving the entire backward SDE is required to complete one training iteration. Therefore, our novel sample-wise back-propagation method for SNNs will be a feasible tool for the uncertainty quantification of deep learning. Although the sample-wise back-propagation procedure does not produce accurate numerical solutions for backward SDEs, we want to point out that the solutions of the adjoint equation are only used to construct the gradient process for gradient descent optimization, and obtaining complete numerical solutions for backward SDEs is not necessary in the back-propagation procedure. 

The contribution of this work has a theoretical aspect and a numerical implementation aspect. Our theoretical contribution will focus on deriving convergence analysis for the sample-wise back-propagation algorithm. Different from classic convergence results for stochastic gradient descent (SGD) based optimization, the main challenge in the convergence analysis for our sample-wise back-propagation solver is that the gradient process is formulated by (discrete) approximated solutions for the adjoint backward SDE, and classic numerical methods for solving backward SDEs require obtaining complete numerical approximations for the solutions in the entire state space \cite{Zhang_BSDE}. However, a sample-wise approximation (or an approximation on mini-batch) for the solution does not carry enough information to deduce convergence of the numerical solver with standard techniques for backward SDEs. Our strategy to address this issue is to introduce an augmented $\sigma$-algebra that contains the uncertainty information in both the SNN model and the training data. After proving some basic properties for the sample-wise approximators, we can derive that the stochastic approximation for the gradient conditioning on the augmented $\sigma$-algebra is an unbiased estimator.  This result would help us obtain our main convergence theorems without usage of general convergence of numerical solvers for backward SDEs. In this paper, we first assume that the optimization problem for searching the optimal SNN parameters (or equivalently the optimal control) is convex, which is a classic assumption in most existing numerical analysis results for stochastic optimal control problems. Our analysis under the convexity assumption shows that the sample-wise back-propagation algorithm will result in two (mean square) error terms: The first error term gives half-order convergence with respect to the depth of neural networks; the second error term is a quotient of the depth of SNN and the number of training steps.  While the half-order convergence with respect to the depth indicates convergence of the discrete SNN model in approximating the continuous SDE formulation of probabilistic learning, the quotient error term can provide inherent relation between the depth of a neural network and the number of iterations needed in the training procedure. Although several analysis outcomes have been derived to study convergence of training with respect to SGD iterations, very limited research is conducted to explore the inherent connection between the depth of a neural network and the number of training iterations that is needed. Moreover, since the convexity assumption may not hold in many practical machine learning applications, we will also derive a convergence theorem without assuming the optimization problem is convex. Specifically, we will show that for an SNN with a fixed depth, our proposed sample-wise back-propagation algorithm will converge with respect to training steps.   

In addition to the effort on deriving theoretical analysis results, in this paper we shall carry out various numerical experiments to examine the performance of our sample-wise back-propagation method, and we will adopt well established DNN structures to the SNN model when solving practical machine learning problems. Specifically, besides demonstration examples of synthetic mathematical problems, which are designed to showcase the baseline performance and convergence behavior of our algorithm, we will also solve the benchmark classification problem with MNIST handwritten dataset and Fashion-MNIST dataset. To enhance the learning capability of SNNs to accomplish the classification task, convolution blocks are incorporated into the SNN model in the numerical implementation. Other DNN structures can be adopted in a similar way to solve machine learning problems in different scenarios. Through our numerical experiments, we will validate the theoretical results that we derived in this paper, and we will present the stability/robustness advantage of SNN (as a probabilistic learning method) compared with the conventional deterministic convolution neural network.

\vspace{0.5em}
The rest of this paper is organized as follows: In Section \ref{Preliminary}, we formulate the training procedure for SNNs as a stochastic optimal control problem and briefly introduce our sample-wise back-propagation method. The main convergence theorems and their proofs will be provided in Section \ref{Convergence}. In Section \ref{Numerics}, we will validate our analysis results and examine the performance of SNNs through several numerical experiments. 

\section{A sample-wise back-propagation method for stochastic neural networks}\label{Preliminary}

In this section, we introduce our sample-wise back-propagation method for training stochastic neural networks. We first formulate the training procedure for the continuous formulation of a stochastic neural network as a stochastic optimal control problem. Then we solve this stochastic optimal control problem via a stochastic gradient descent algorithm. 

\subsection{Stochastic optimal control formulation for SNN}

We consider the following dynamical system for a stochastic version of deep neural network (DNN):
\begin{equation}\label{SNN-discrete}
X_{n+1} = X_n + h f(X_n, u_{n})  + \sqrt{h} g(u_{n}) \omega_n, \qquad n = 0, 1, 2, \cdots, N-1,
\end{equation}
where $X_n: = [x_n^1, x_n^2, \cdots, x_n^L] \in \mathbb{R}^L$ is a vector that contains $L$ neurons at the $n$-th layer in a DNN, $b$ is an activation function, $u_n$ denotes the set of DNN parameters at the $n$-th layer, $h$ is a positive constant that stabilizes the DNN, $\omega_n$ is a standard $L$-dimensional Gaussian random variable that brings uncertainty to the neural network, and $g$ is a coefficient function that determines the size of uncertainty in the DNN. Since the uncertainty of the DNN is incorporated by the Gaussian random variables, we assume that parameters $\{u_n\}_n$ are deterministic. The initial state $X_0$ of the dynamical system \eqref{SNN-discrete} represents the input layer, and $X_N$ is the output layer, which will be compared with data. We call the noise perturbed DNN model \eqref{SNN-discrete} a stochastic neural network (SNN) in this paper. 

If we choose a positive constant $T$ (as a terminal time) and let $N \rightarrow \infty$ or $h \rightarrow 0$ with $h = \f{T}{N}$, the dynamics of SNN \eqref{SNN-discrete} becomes the  stochastic differential equation (SDE) which, in the integral form, is given by 
\begin{equation}\label{SNN-continuous}
X_T = X_0+ \int_{0}^T f(X_t, u_t) dt  + \int_{0}^T g(u_t) dW_t,
\end{equation}
where $W:= \{W_t\}_{0 \leq t \leq T}$ is a standard Brownian motion corresponding to the i.i.d. Gaussian random variable sequence $\{w_n\}_n$ in \eqref{SNN-discrete},  $\int_{0}^T g(u_t) dW_t$ is an It\^o integral.  $X_T$ is now corresponding to the output layer. Let $\Gamma$ be the random variable that generates the training data to be compared with $X_T$ and define  a loss function $\Phi(X_T, \Gamma): = \|X_T - \Gamma\|_{loss}$ corresponding to a loss error norm $\|\cdot\|_{loss}$. The goal of the training procedure for deep learning is to determine the optimal parameter to minimize the loss function.

In this paper, we shall treat the training procedure for deep learning as a stochastic optimal control problem and the parameter $u_t$ in \eqref{SNN-continuous} as a control process.  To this end,  we define  the cost functional $J$ as 
\begin{equation}\label{eq:cost_functional}
J(u) = \E\left[ \int_{0}^T r(X_t, u_t) dt + \Phi(X_T, \Gamma) \right],
\end{equation}
where the integral $\int_{0}^T r(X_t, u_t) dt$ represents the running cost in a control problem.  In this way, the loss function becomes the terminal cost in the stochastic optimal control problem. Note that the stochastic process $X$ is the ``state process'' for the stochastic optimal control problem which depends on the control $u$.  The goal of deep learning is to solve the stochastic optimal control problem, i.e.,  find the optimal control $u^*$ such that 
\begin{equation}\label{control-cost}
J(u^*) = \inf_{u \in \mathcal{U}[0, T]} J(u),
\end{equation}
where $\mathcal{U}[0, T] : = \{u \in L^2\big([0, T]; \mathbb{R}^p \big) \}$ is a $p$-dimensional admissible control set. 

In order to determine the optimal control, one can derive the following gradient process \cite{Yong_BSDE, Yong_control}:
\begin{equation}\label{gradient-continuous}
\nabla J_u(u_t) =  \E\big[ f_u(X_t, u_t)^\top Y_t + g_u(u_t)^\top Z_t + r_u(X_t, u_t)^\top\big],
\end{equation}
where $f_u$, $g_u$ and $r_u$ are partial derivatives with respect to the control $u$. In this paper, we use the notation $\nabla J_u$ to denote the gradient of the cost $J$ with respect to the control process $u$, and the control at time $t$, i.e., $u_t$, in $\nabla J_u(\cdot)$ indicates that we are considering the gradient process at time $t$. 
Here the stochastic processes $Y_t$ and $Z_t$ in \eqref{gradient-continuous}  are the adapted solutions of the following backward SDE:
\begin{equation}\label{BSDE-continuous}
dY_t = \left(- f_x(X_t, u_t)^\top Y_t - r_x(X_t, u_t)^\top\right)dt + Z_t dW_t, \qquad Y_T = \Phi'_x(X_T, \Gamma),
\end{equation}
where $f_x$, $r_x$ are partial derivatives with respect to the state $X$, $Y$ is the adjoint process of the state SDE $X$, and $Z$ is the martingale representation of $Y$ with respect to $W$. An important property of solutions of the backward SDE is that values of $Y$ and $Z$ depend on $X$, and $(Y_t, Z_t) \in \mathcal{F}^W_t$, where $\mathcal{F}_t^W := \sigma(W_s, 0 \leq s \leq t)$. 

Then, one can carry out gradient descent optimization to determine the optimal control by using the gradient $\nabla J_u$ introduced in \eqref{gradient-continuous} as follows:
\begin{equation}\label{GD-continuous}
u^{k+1}_t = \mathcal{P}_{\mathcal{U}}\big( u^{k}_t - \eta_k \nabla J_u(u_t^k)\big), \qquad k = 0, 1, 2, \cdots, \quad 0 \leq t \leq T,
\end{equation}
where $u^0$ is an initial guess for the optimal control, $\eta_k$ is the step-size of gradient descent in the $k$-th iteration step, and $\mathcal{P}_{\mathcal{U}}$ is a projection operator onto the admissible control set $\mathcal{U}$.

Numerical implementation of the gradient descent optimization scheme \eqref{GD-continuous} requires numerical approximations for solutions $X$, $Y$ and $Z$ of the backward SDE \eqref{BSDE-continuous}. In the following subsection, we shall first introduce numerical schemes to approximate $Y$ and $Z$, and then utilize stochastic approximation to efficiently carry out a stochastic gradient descent procedure to determine the optimal control.

\subsection{Numerical schemes for backward SDEs and stochastic gradient descent optimization}\label{Algorithm}
We solve the backward SDE  \eqref{BSDE-continuous} over a uniform temporal partition $\Pi_N :=\{0 = t_0 < t_1 < t_2 < \cdots < t_N = T\}$. Denote by $h$ the stepsize of the partition.  On each subinterval $[t_n, t_{n+1}]$, $n=0, 1, \cdots, n-1$, the  backward SDE  satisfies 
\begin{equation}\label{BSDE-discrete}
Y_{t_{n}} = Y_{t_{n+1}} + \int_{t_n}^{t_{n+1}}f_x(X_t, u_t)^\top Y_t + r_x(X_t, u_t)^\top dt - \int_{t_n}^{t_{n+1}} Z_t dW_t.
\end{equation}
We introduce the following Euler-Maruyama scheme to approximate the forward SDE \eqref{SNN-continuous} on interval $[t_n, t_{n+1}]$:
\begin{equation}\label{SDE-discrete}
X_{t_{n+1}} = X_{t_n} + f(X_{t_n}, u_{t_n}) h + g(u_{t_n}) \Delta W_{t_n} + R_{n}^X,
\end{equation}
where $\Delta W_{t_n} : = W_{t_{n+1}} - W_{t_n}$, and $R_{n}^X$ denotes the  approximation errors of deterministic and stochastic integrals.

To approximate solution $Y_{t_n}$, we take conditional expectation $\E_n^X[\cdot] : = \E[\cdot \big| X_{t_n} = X]$ on both sides of \eqref{BSDE-discrete} and approximate the deterministic integral by using the right-point formula to get
\begin{equation}\label{E:BSDE-discrete}
Y_{t_{n}}|_{_{X_{t_n} = X}} = \E_n^X[Y_{t_{n+1}}] + h \E_n^X[ f_x(X_{t_{n+1}}, u_{t_{n+1}})^\top Y_{t_{n+1}} + r_x(X_{t_{n+1}}, u_{t_{n+1}})^\top] + R_{n}^Y,
\end{equation}
where $R_{n}^Y$ is the approximation error, and we denote $Y_{t_{n}}|_{_{X_{t_n} = X}} = \E_n^X[Y_{t_{n}}]$ since $Y_{t_n}$ is $\mathcal{F}^W_{t_n}$ measurable. Note that the stochastic integral $\int_{t_n}^{t_{n+1}} Z_t dW_t$ is eliminated under expectation. 

To get an approximation for $Z_{t_n}$, we multiply $\Delta W_{t_n}$ on both sides of \eqref{BSDE-discrete} and then take conditional expectation $\E_n^X[\cdot]$ to get
\begin{equation*}
\begin{aligned}
0 = \E_n^X[ Y_{t_{n+1}} \Delta W_{t_n}] + \E_n^X\Big[\int_{t_n}^{t_{n+1}} & f_x(X_t, u_t)^\top Y_t + r_x(X_t, u_t)^\top dt \Delta W_{t_n} \Big] \\
& - \E_n^X\big[ \int_{t_n}^{t_{n+1}} Z_t dW_t\Delta W_{t_n} \big].
\end{aligned}
\end{equation*}
We use the left-point formula to approximate both the deterministic and the stochastic integral in the above equation and obtain
\begin{equation}\label{E:Z-discrete}
Z_{t_n}|_{_{X_{t_n} = X}} = \E_n^X[ \f{Y_{t_{n+1}} \Delta W_{t_n} }{h}] + R_{n}^Z,
\end{equation}
where $R_n^Z$ is the approximation error. 

By dropping the approximation error terms $R_n^X$, $R_n^Y$ and $R_n^Z$ in \eqref{SDE-discrete}, \eqref{E:BSDE-discrete} and \eqref{E:Z-discrete}, respectively, we get the following numerical schemes to solve the backward SDE \eqref{BSDE-continuous} (see \cite{Zhang_BSDE, Zhao_BSDE, Bao_first, Bao_Zakai, BCZ_2011, BCZ_2015, BCZ_2018, BSDE_filter} for more details): 
\begin{equation}\label{scheme:X}
X^{N}_{n+1} = X^{N}_n + h f(X^{N}_n, u_{t_n}) + g(u_{t_n}) \Delta W_{t_n},
\end{equation}
\begin{equation}\label{schemes:YZ}
\begin{aligned}
Y^{N}_{n} =& \ \E_n^X[Y^{N}_{n+1}] + h \E_n^X\big[ f_x(X^{N}_{n+1}, u_{t_{n+1}})^\top Y^{N}_{n+1} + r_x(X^{N}_{n+1}, u_{t_{n+1}})^\top \big], \\
Z^{N}_n =& \ \E_n^X\Big[ \f{Y^{N}_{n+1} \Delta W_{t_n}}{h}\Big],
\end{aligned}
\end{equation}
where $X^{N}_{n+1}$, $Y^{N}_n$ and $Z^{N}_n$ are numerical approximations for $X_{t_{n+1}}$, $Y_{t_n}$ and $Z_{t_n}$, respectively. The reason that we can remove the conditional expectation for $Y^{N}_{n}$ and $Z^{N}_n$ on the left hand side of Eq. \eqref{schemes:YZ} is due to that fact that $Y_{t_n}$ and $Z_{t_n}$ are adapted to $\mathcal{F}^W_{t_n}$. 
Then, we replace $X_{t_n}$, $Y_{t_n}$ and $Z_{t_n}$ in \eqref{gradient-continuous} with our numerical solutions and obtain a piece-wise approximation for $\nabla J_u(t, u_{t})$ over the interval $[t_n, t_{n+1})$ as follows:
\begin{equation}\label{gradient-discrete}
\nabla J_u(t, u_{t}) \approx  \nabla J^N_u(t_n, u_{t_n}) : = \E\big[ f_u(X^N_{n}, u_{t_n})^\top Y^{N}_{n} + g_u(u_{t_n})^\top Z^{N}_{n} + r_u(X^N_{n}, u_{t_n})^\top\big]. 
\end{equation}

To calculate the gradient through scheme \eqref{gradient-discrete}, we need to evaluate the expectation $\E[\cdot]$ (with respect to the state process $X$).  In addition, the approximate solutions $Y_n$ and $Z_n$ are under the conditional expectation $\E_n^X[\cdot]$. The conventional approach for simulating expectations is the Monte Carlo method. However, when using stochastic optimal control to formulate the training procedure of the SNN, the dimension of the state $X$ is equivalent to the number of neurons. Hence it requires  very large number of Monte Carlo samples to approximate expectations, and the computational cost for calculating the gradient in each gradient descent iteration step is very high. Moreover, note that   $Y$ and $Z$ depend on $X$, which is typically represented by a set of pre-chosen points in the state space. Therefore, interpolation methods are needed to calculate $Y_n$ and $Z_n$ at random state points in order to implement the Monte Carlo simulation. This is a very challenging task due to ``curse of dimensionality''.

To address the above challenges, we apply the stochastic approximation method to approximate expectations in the approximation scheme \eqref{gradient-discrete} for the gradient $\nabla J_u$. Specifically, at the $k$-th iteration step with an estimated optimal control $u^k$, we simulate the state process by choosing a sample $\omega^k \sim N(0, h)$ and get
\begin{equation}\label{SDE-sample}
X^k_{n+1} = X^k_n + h f(X^k_n, u^k_{t_n}) + g(u^k_{t_n}) \omega_n^k.
\end{equation}
Instead of using the average of a large number of Monte Carlo samples, we approximate the expectations in schemes \eqref{schemes:YZ} with the single-realization of the simulate sample path $\{X^k_n\}_{n=1}^N$ and obtain the following sample-wise solution for $Y$ and $Z$:
\begin{equation}\label{YZ:sample}
Y^k_{n} = Y^k_{n+1} + h \big( f_x(X^k_{n+1}, u^k_{t_{n+1}})^\top Y^k_{n+1} + r_x(X^k_{n+1}, u^k_{t_{n+1}})^\top\big), \  Z^k_n = \f{Y^k_{n+1} \omega_n^k }{h},
\end{equation}
where $Y^k_n$ and $Z^k_n$ are approximations for $Y_{t_n}|_{X_{t_n} = X^k_{n}}$ and $Z_{t_n}|_{X_{t_n} = X^k_{n}}$, respectively. Although schemes \eqref{YZ:sample} do not produce numerical solutions for $Y$ and $Z$ in the entire state space,  the solutions $Y$ and $Z$ of the backward SDE \eqref{BSDE-continuous} \textit{only appear in the gradient process}, and complete numerical solutions for $Y$ and $Z$ in the entire state space are not necessary in the gradient descent optimization procedure for the optimal control. In this way, the justification for application of stochastic approximation in stochastic gradient descent can be used to explain the design of the sample-wise approximation scheme \eqref{YZ:sample}. Therefore, we approximate the gradient $\nabla J_u$ by the following scheme
\begin{equation}\label{gradient-sample}
\nabla j^k_u(u^k_{n}) :=  f_u(X^k_{n}, u^k_{n})^\top Y^k_{n} + g_u(u^k_{n})^\top Z^k_{n} +  r_u(X^k_{n}, u^k_{n})^\top,
\end{equation}
and we carry out stochastic gradient descent (SGD) optimization to search for the optimal control $u^*$ as follows:
\begin{equation}\label{SGD-u}
u^{k+1}_{t_n} = \mathcal{P}_{\mathcal{U}_N}\big( u^{k}_{t_n} - \eta_k \nabla j^k_u(u^k_{t_n})\big), \qquad k = 0, 1, 2, \cdots, \quad 0 \leq n \leq N,
\end{equation}
where $\mathcal{U}_N : = \mathcal{U} \cap \mathcal{C}_N$ with a piece-wise constant approximation set $\mathcal{C}_N : = \{u| u = \sum_{n=0}^{N-1} a_n 1_{[t_n, t_{n+1})}, a_n \in \mathbb{R}^p\}$ for the control process. 

The above schemes \eqref{SDE-sample} - \eqref{SGD-u} provide  an SGD algorithm for solving the stochastic optimal control problem, and it also constitutes a sample-wise back-propagation framework for training the SNN. It's worthy to mention that in machine learning practice, the above sample-wise back-propagation scheme is often implemented by using a mini-batch of samples instead of a single-realization of sample to represent the state variable. For convenience of presentation, we shall use the single sample representation in the rest of this paper. The numerical implementation and analysis for the mini-batch case can be derived following a similar strategy.

\section{Convergence Analysis}\label{Convergence}
In this section, we conduct the convergence analysis  for the SGD algorithm \eqref{SDE-sample} - \eqref{SGD-u}, and we have the following standard assumptions for backward SDEs and the stochastic optimal control problem  \cite{Zhao_BSDE, Yong_control}:
\begin{ass}\label{Assumptions}
\begin{itemize}
\item[(a)]  Both $f=f(x,u)$ and $g=g(x,u)$  belong to $C_b^{2,2}(\mathbb{R}^d \times \mathbb{R}^p; \mathbb{R}^d)$.
\item[(b)] $f$, $f_x$, $f_u$, $g$, $g_u$, $r_x$, $r_u$ are uniformly bounded and Lipschitz in $x$ and $u$.
\item[(c)] $g$ satisfies the uniform elliptic condition.
\item[(d)] The initial condition $X_0 \in L^2(\mathbb{R}^d)$.
\item[(e)] The loss function $\Phi$ is $C^1$ and positive, and $\Phi'_x$ has at most linear growth at infinity.
\item[(f)] $\lim_{\|u\|_2 \rightarrow \infty} J(u) = \infty$. 
\vspace{0.5em}
\end{itemize}
\end{ass}
We define the standard  inner product of $u, v \in L^2([0, T])$ by $\langle u, v\rangle = \int_{0}^T u_t \cdot v_t dt$ and the  standard $L^2$ norm $\|\cdot\|_2$ by $\|u\|_2:=\sqrt{\langle u, u\rangle}$. Without confusion, we use the same notation to denote the  inner product of two piece-wise constant representations $u^N$, $v^N\in \mathbb{R}^N$ of $u$, $v$:  $\langle u^N, v^N \rangle = h \sum_{n=0}^{N-1}u_n^N \cdot v_n^N$. Also, we use $|\cdot|$ to denote the Euclidean norm or Frobenius norm.

\vspace{0.5em}
Our analysis in this paper will focus on the convergence with respect to the temporal partition $\Pi_N $. Convergence regarding the number of neurons, i.e., the dimension of the state $X$ in SDE \eqref{SNN-continuous}, may bring out more complicated discussions on spatial dimension approximation and representability of neural networks, which is out of scope of this paper.  Specifically, we will derive convergence theorems under two scenarios.  First, we assume that the cost functional  for the optimal control is convex and prove a half-order convergence rate for our algorithm. In the second situation, we consider the case without convexity assumption and provide a convergence result for our sample-wise back-propagation method. 

\subsection{Preliminaries}
The foundation of the convergence analysis  is based on the fact that the sample-wise solutions $Y^{k}_n$ and $Z^{k}_n$ introduced in \eqref{YZ:sample} are equivalent to the classic numerical solutions $Y_n^N$ and $Z_n^N$ introduced in \eqref{schemes:YZ} under conditional expectation $\E_n^X[\cdot]$. Specifically, we have the following proposition. 
\begin{prop}\label{Y^k=Y^N}
For given estimated control $u^k \in \mathcal{U}_N$, let $Y_n^{k, N}$ and $Z_n^{k, N}$ be the numerical solutions defined in \eqref{schemes:YZ} driven by $u^k$.  Then following identities hold:
\begin{equation}\label{Y=Y}
\E_n^X[Y_n^k] = Y_n^{k, N}|_{_{X_{n} = X}}, \qquad \E_n^X[Z_n^k] = Z_n^{k, N}|_{_{X_{n} = X}}, \qquad 0 \leq n \leq N-1, 
\end{equation} 
and therefore we have $\E[Y_n^k] = \E[Y_n^{k, N}]$ and $\E[Z_n^k] = \E[Z_n^{k, N}]$.
\end{prop}
\begin{proof}
Note that the random variable $\omega_n^k$ in scheme \eqref{SDE-sample} has the same distribution of $\Delta W_{t_n}$ appeared in \eqref{scheme:X}. Hence, the random variable $\omega_n^k$ is equivalent to $\Delta W_{t_n}$ under expectation.
More generally, for any function $\phi(\{\omega_n^k\}_n)$ of the random sample path $\{\omega_n^k\}_n$, we have $\E[\phi(\{\omega_n^k\}_n)] = \E[\phi(\{\Delta W_{t_n}\}_n)]$.

To obtain the desired results in the proposition, we first consider the case $n = N-1$ (i.e. take one step back from the terminal time), and we have
\begin{equation*}
\E_{N-1}^X [Z_{N-1}^k] = \E_{N-1}^X \left[\f{Y_N^{k} \omega_{N-1}^k}{h}\right] = \E_{N-1}^X \left[\f{Y_N^{k, N} \Delta W_{t_{N-1}}}{h}\right] = Z_{N-1}^{k, N}|_{_{X_{N-1} = X}}.
\end{equation*}
Following the same argument, we also have
\begin{equation*}
\begin{aligned}
\E_{N-1}^X [Y_{N-1}^k] =  \E_{N-1}^X\left[ Y^k_{N} + h \big( f_x(X^k_{N}, u^k_{t_{N}})^\top Y^k_{N} + r_x(X^k_{N}, u^k_{t_N})^\top\big)\right].
\end{aligned}
\end{equation*}
Since $Y^k_N = \Phi_x = Y_N^{k, N}$ and $\E_{N-1}^X[X^{k}_N] = \E_{N-1}^X[X^{k, N}_N]$, where $X^{k, N}_N$ is the approximated solution introduced in \eqref{scheme:X} with the given control $u^k$, the above equation becomes
\begin{equation}\label{EYk=EYN}
\begin{aligned}
\E_{N-1}^X [Y_{N-1}^k] =  \E_{N-1}^X\left[ Y^{k, N}_{N} + h \big( f_x(X^{k,N}_{N}, u^k_{t_{N}})^\top Y^{k, N}_{N} + r_x(X^{k,N}_{N}, u^k_{t_N})^\top\big)\right] = Y_{N-1}^{k, N}|_{_{X_{N-1} = X}}.
\end{aligned}
\end{equation}
Then, by repeatedly applying the equality \eqref{EYk=EYN} and the tower property, we  obtain the desired result \eqref{Y=Y}.
\end{proof}

To proceed, we let $\mathcal{G}_k := \sigma(\omega^i, \gamma^i, 0 \leq I \leq k-1)$ be an augmented $\sigma$-algebra generated by the Gaussian random variables $\omega^i$, which we use to generate state sample path $X^k$ in the sample-wise scheme \eqref{SDE-sample}, and the data sample $\gamma^i$, which is generated by the training data variable $\Gamma$. Based on the above proposition, we see  that the stochastic approximation $\nabla j^k_u(u^k_{t_n})$ that we introduced in \eqref{gradient-sample} is an unbiased estimator for the gradient $\nabla J^N_u(u^k_{t_n})$ given $\mathcal{G}_k$, i.e. $\E[\nabla j^k_u (u^k_{t_n})\big|\mathcal{G}_k]  = \nabla J^N_u(u_{t_n}^k)$. From scheme \eqref{SGD-u}, we observe that the estimated optimal control $u^k$ is $\mathcal{G}_k$ measurable, and we denote $\E^k[\cdot] : = \E[\cdot \big|\mathcal{G}_k]$ in the rest of this paper for convenience of presentation. Therefore, similar to the argument in the proof for Proposition \ref{Y^k=Y^N}, we have 
\begin{equation*}
\begin{aligned}
\E^k[\nabla j^k_u (u^k_{t_n})] =& \ \E^k\left[ f_u(X^k_{n}, u^k_{n})^\top Y^k_{n} + g_u(u^k_{t_n})^\top Z^k_{n} +  r_u(X^k_{n}, u^k_{t_n})^\top \right]\\
=& \  \E^k\left[ f_u(X^{k, N}_{n}, u^k_{t_n})^\top Y^{k, N}_{n} + g_u(u^k_{t_n})^\top Z^{k, N}_{n} +  r_u(X^{k, N}_{n}, u^k_{t_n})^\top  \right] \\
= & \ \nabla J^N_u(u_{t_n}^k).
\end{aligned}
\end{equation*}

In the following lemma, we provide the boundedness property of the sample-wise solution $Y^k$ and the linear growth property for $Z^k$ with for any approximate  control $u^k \in \mathcal{U}_N$. 
\begin{lem}\label{Lem:boundedness}
Under Assumptions (a)-(e), for any $u^k \in \mathcal{U}_N$, we have
$$\sup_{0\leq n \leq N} \E[(Y^k_n)^2] \leq C, \qquad \sup_{0\leq n \leq N} \E[(Z^k_n)^2] \leq C N $$
for some positive constant  $C$.
\end{lem}
\begin{proof}
We square both sides of the scheme for $Y^k_n$ defined in  \eqref{YZ:sample} and  then take expectation to get
$$
\begin{aligned}
\E[(Y^k_n)^2] \leq & \ \E\left[(1+h)(1 + h f_x(X^k_{n+1}, u^k_{t_{n+1}})^\top) (Y^k_{n+1})^2\right] + (1 + \f{1}{h}) \E\left[(r_x(X^k_{n+1}, u^k_{t_{n+1}})^\top)^2 h^2\right] \\
\leq & (1 + C h) \E\left[(Y^k_{n+1})^2\right] + C h,
\end{aligned}
$$
where we have used the Young's inequality.
Then, by the discrete Gronwall inequality, we have
$$\sup_{0 \leq n \leq N} \E[(Y^k_n)^2] \leq C \E[(Y^k_N)^2] + C.$$
Due to the Assumption \ref{Assumptions} (e) for the loss function $\Phi$, we have that $\sup_{0 \leq n \leq N} \E[(Y^k_n)^2] \leq C$ as desired.

Next, we  square both sides of the scheme for $Z^k_n$ and take expectation. Based the boundedness for $Y^k_n$, we have that
$$
\begin{aligned}
\E[(Z^k_n)^2] \leq & \ \E\left[\left( \f{Y^k_n  \omega^k_{n}}{h}\right)^2\right] \leq C N,
\end{aligned}
$$
which is the desired boundedness for $Z^k_n$. 
\end{proof}

In addition, we have the following classic boundedness property for solutions of backward SDEs (see Theorem 4.2.1 in \cite{Zhang_BSDE}): 
\begin{equation}\label{bound:YZ:1}
\sup_{0\leq t \leq T} \E[|Y_t|^2] + \E\left[\int_{0}^{T}|Z_t|^2 dt \right] \leq C,
\end{equation}
which will provide the following boundedness result
\begin{equation}\label{bound:YZ:2}
\sup_{0\leq t \leq T} \E[|Y^k_t|^2] + \E\left[\int_{0}^{T}|Z^k_t|^2 dt \right] \leq C
\end{equation}
for the analytic solutions $Y_t^k$ and $Z^k_t$ driven by a control $u^k \in \mathcal{U}_N$, where the constant $C$ is independent of  $u^k$. Following the above estimates, one can also show that $\sup_{0\leq t \leq T} \E[|Z^k_t|^2] \leq C$ under Assumption \ref{Assumptions} (a)-(e). 

Note that the boundedness of the sample-wise solution $Z^k_n$ in the estimate in Lemma \ref{Lem:boundedness} is not as strong as the true solution $Z^k_t$  due to the loss of regularity caused by the sample-wise representation of expectation introduced in \eqref{YZ:sample}.

The next proposition is about the convergence of the numerical solutions to backward SDEs introduced in \eqref{schemes:YZ}  also holds:
\begin{prop}\label{Convergence-Classic}
Under Assumption \ref{Assumptions}, for small enough step-size $h$, there exists $C > 0$ independent of $u^k \in \mathcal{U}_N$ such that
\begin{equation}\label{YZ:Convergence}
\max_{0\leq n \leq N}\E\left[ \sup_{t_n} |Y^k_t - Y^{k, N}_{n}|^2\right] + \sum_{n=1}^{N-1}\E\left[ \int_{t_n}^{t_{n+1}}|Z^k_t - Z^{k, N}_{n}|^2 dt \right] \leq C (1 + |X_0|^2) h.
\end{equation}
\end{prop}
\begin{proof}
The inequality \eqref{YZ:Convergence} is a standard extension of the regularity property, and the proof of the theorem can be derived following the proof for Theorem 5.3.3 in \cite{Zhang_BSDE}. The fact that $C$ does not depend on $u^k$ is due to the uniform boundedness and the Lipschitz assumptions of $f$, $g$ and $r$.

\end{proof}

The boundedness property of solutions of backward SDEs gives us the boundedness of $\nabla J_u$. Also, the convergence result \eqref{YZ:Convergence} gives the boundedness of $(Y^{k, N}_{n}, Z^{k, N}_{n})$, which makes $\nabla J^N_u$ bounded, i.e. 
\begin{equation}\label{bound:J'^N}
\sup_{u^k \in \mathcal{U}_N} \| \nabla J^N_u(u^k)\|_2 < C.
\end{equation}
As a consequence of the above discussions, we have the following lemma.
\begin{lem}\label{Error:Gradient}
Under Assumption \ref{Assumptions}, for any piece-wise constant estimated control $u^k \in \mathcal{U}_N$, the following estimation holds
\begin{equation}
\E\left[\|\nabla j^k_u(u^k) - \nabla J^N_u(u^k) \|_2^2 \right] \leq C N.
\end{equation}
\end{lem}
\begin{proof}
Due to Lemma \ref{Lem:boundedness} and the boundedness assumptions for $f_u$, $g_u$ and $r_u$, we have that 
\begin{equation}\label{bound:j'}
|\nabla j^k_u(u^k_{t_n})|^2 \leq C \left( |Y^k_n|^2 + |Z^k_n|\right) \leq C N.
\end{equation}
Then, we can obtain
\begin{equation*}
\begin{aligned}
& \ \E\left[\|\nabla j^k_u(u^k) - \nabla J^N_u(u^k) \|_2^2 \right] \leq  2 \E\left[\|\nabla j^k_u(u^k) \|_2^2 \right] + 2\E\left[\|\nabla J^N_u(u^k) \|_2^2 \right] \\
\leq & \ C N + C h \sum_{n=0}^{N-1}\E\left[ |f_u(X_n^{k, N}, u^k_{t_n})^\top Y^{k, N}_n|^2 + |g_u(X_n^{k, N}, u^k_{t_n})^\top Z^{k, N}_n|^2 + |r_u(X_n^{k, N}, u^k_{t_n})^\top|^2\right] \\
\leq & \ C N + C h \sum_{n=0}^{N-1}\sup_{0 \leq n \leq N-1}\E\left[ |Y^{k, N}_n|^2 + |Z^{k, N}_n|^2 \right] + C \\
\leq & \ C N + C,
\end{aligned}
\end{equation*}
where $C>0$ is a generic constant that is independent from $N$. Hence we can get the desired result of the lemma from the above analysis. 
\end{proof}

\subsection{Convergence analysis: strongly convex cost functional}\label{Section: Convex}
Here, we assume that the cost functional $J$ defined in \eqref{eq:cost_functional} is strongly convex in the following sense: there exists some constant $\lambda > 0$  such that for any control terms $u, v \in \mathcal{U}$, \begin{equation}\label{Convexity_Assumption}
\langle \nabla J_u(u) - \nabla J_u(v), u - v \rangle \geq \lambda \|u - v\|_2^2.
\end{equation}
By the Assumption \ref{Assumptions},   we have the  following smoothness result for $J$ with a positive constant $C_L$:
$$\int_0^T |\nabla J_u(u_t) - \nabla J_u(v_t)|^2 dt \leq C_L \int_0^T |u_t - v_t|^2 dt,$$
or equivalently,
\begin{equation}\label{Convexity}
\|\nabla J_u(u) - \nabla J_u(v)\|_2^2 \leq C_L\|u - v\|_2^2.
\end{equation}

Before stating  the main convergence theorem, we need  an estimate  about the error between the true gradient $\nabla J_u$ and the piece-wisely approximated gradient $\nabla J_u^N$.
\begin{lem}\label{lemma-convex}
Assume that Assumption \ref{Assumptions} holds and $f, g \in C_b^4$, $u^k \in \mathcal{U}_N$. Then there exists a constant $C > 0$ such that 
\begin{equation}\label{estimate:J'}
\sup_{u^k \in \mathcal{U}_N}\|\nabla J_u(u^k) - \nabla J^N_u(u^k)  \|_2^2 \leq \f{C}{N}.
\end{equation}
\end{lem} 
\begin{proof}  Denote
$$
\begin{aligned}
\phi^k_{t} := & \ f_u(X^{k}_{t}, u^{k}_t)^\top Y^k_{t} + g_u(X^{k}_{t}, u^{k}_t)^\top Z^k_{t} +r_u(X^{k}_{t}, u^{k}_t)^\top, \\
\phi^k_{n} := & \ f_u(X^{k, N}_{n}, u^{k}_{t_n})^\top Y^{k, N}_{n} + g_u(X^{k, N}_{n}, u^{k}_{t_n})^\top Z^{k, N}_{n} +r_u(X^{k, N}_{n}, u^{k}_{t_n})^\top. 
\end{aligned}
$$
We have that 
$$
\begin{aligned}
& \int_{0}^{T} ( \nabla J_u(u^k_t) - \nabla J^N_u(u^k_t) )^2 dt\\
\leq & \ 2 \sum_{n=0}^{N-1}\int_{t_n}^{t_{n+1}} \left[ (\nabla J_u(u^k_t) - \nabla J_u(u^k_{t_n}))^2 + (\nabla J_u(u^k_{t_n}) - \nabla J^N_u(t_n, u^k_{t_n}))^2 \right] dt \\
= & \ 2 \sum_{n=0}^{N-1} \int_{t_n}^{t_{n+1}}(\E[\phi^k_t - \phi^k_{t_n}])^2 dt + 2 \sum_{n=0}^{N-1} \int_{t_n}^{t_{n+1}} (\E[\phi^k_{t_n} - \phi^k_{n}])^2 dt \leq \f{C}{N}.
\end{aligned}
$$ 
\end{proof}

Let $u^{*, N}$ be the optimal control for the stochastic optimal control problem \eqref{control-cost} found in the subset  $\mathcal{U}_N$ of the admissible control set $\mathcal{U}$. 
Then, we have the following estimate. 
\begin{lem}
The following inequality holds: 
\begin{equation}\label{Order:J'}
\| \nabla J_u(u^{*, N}) - \nabla J_u(u^*)\|_2^2 \leq \f{C}{N}.
\end{equation}
\end{lem}
\begin{proof}
Let $\bar{u}^N$ be the projection of $u^*$ onto $\mathcal{U}_N$, i.e. $\bar{u}^N = \mathcal{P}_{\mathcal{U}_N}(u^*) = \arg\min_{u^N \in \mathcal{U}_N} \|u^N - u^* \|_2$. Due to the boundedness of $u^*$, we have 
\begin{equation*}
\| \bar{u}^N - u^* \|_2 \leq \f{C}{N}.
\end{equation*}
Since the solutions of the backward SDEs are bounded, $\nabla J_u$ is also bounded. Thus \begin{equation*}
\begin{aligned}
J(\bar{u}^N) - J(u^*) = & \ \int_0^1 \langle \nabla J_u\big(u^* + \epsilon (\bar{u}^N - u^*) \big),  \bar{u}^N - u^*)\rangle d \epsilon \\
\leq & \sup_{u \in L^2} \| \nabla J_u (u)\|_2 \|\bar{u}^N - u^*\|_2 \leq C \|\bar{u}^N - u^*\|_2 \leq \f{C}{N}.
\end{aligned}
\end{equation*}
Hence 
$$J(u^{*, N}) - J(u^*) \leq J(\bar{u}^N) - J(u^*) \leq \f{C}{N}. $$
By the strong convex assumption, we have
$$\langle \nabla J_u(u^*), u^{*, N} - u^*\rangle + \f{\lambda}{2}  \| u^{*, N} - u^* \|_2^2 \leq J(u^{*, N}) - J(u^*).$$
Since $\nabla J_u(u_t^*) = 0$, the above inequality leads to
\begin{equation}\label{u^N-u^*}
 \| u^{*, N} - u^* \|_2^2 \leq \f{C}{N}.
 \end{equation}
The desired result of the lemma is obtained by the above estimate and the convexity assumption  \eqref{Convexity}. 
\end{proof}

Since  $\nabla J_u(u_t^*) = 0$,  as a direct consequence of  \eqref{Order:J'}, we have 
\begin{equation}\label{J'<C/N}
\| \nabla J_u(u^{*, N})\|_2^2 \leq \f{C}{N}.
\end{equation}

Now we  are ready to prove the main convergence results of our algorithm under the convexity assumption. First  we  estimate the error between the exact solution of the optimal control and the optimal control in the piece-wise constant subset $\mathcal{U}_N$ of the admissible set $\mathcal{U}$.
\vspace{0.25em}

\begin{thm}\label{Thm: convex_convergence}
Assume all the assumptions in Lemma \ref{lemma-convex} and the convexity assumption are true. Let $\eta_k = \f{\theta}{k+M}$ for some constants $\theta$ and $M$ such that $\lambda \theta - 4 C_L \theta^2 /(1 + M) > 2$. Also, let $\{u^k\}_k$ be the sequence of estimated optimal control obtained by the SGD optimization scheme \eqref{SGD-u}. Then, for large enough $K$, the following inequality holds
\begin{equation}\label{Convex:convergence}
\E\left[\|u^{K+1} - u^{*, N}\|_2^2\right] \leq C \left( \f{N}{K} + \f{1}{N} \right).
\end{equation}
\end{thm}

\begin{proof}
Recall that $u^{*, N}$ is the optimal control found in the control set $\mathcal{U}_N$. Hence 
\begin{equation}\label{u*=u*}
u^{*, N} = \mathcal{P}_{\mathcal{U}_N} (u^{*, N}) = \mathcal{P}_{\mathcal{U}_N} \left(u^{*, N} - \eta_k \nabla J_u(u^{*, N}) + \eta_k \nabla J_u(u^{*, N})\right). 
\end{equation}
Then, we subtract  \eqref{u*=u*} from the SGD optimization scheme \eqref{SGD-u} to get 
$$ \| u^{k+1} - u^{*, N} \|_2^2=\| \mathcal{P}_{\mathcal{U}_N}\left( (u^k - u^{*, N}) - \eta_k \big(\nabla j^k_u(u^k_t) - \nabla J_u(u_t^{*, N}) \big) - \eta_k \nabla J_u(u_t^{*, N})  \right) \|_2^2 $$ 
By taking conditional expectation $\E^k[\cdot]$  to the above equation and then applying the Young's inequality, we obtain
\begin{equation}\label{Analysis-convex}
\begin{aligned}
& \E^k\big[ \| u^{k+1} - u^{*, N} \|_2^2 \big] \\
\leq & \ (1 + \epsilon)\E^k\left[ \| (u^k - u^{*, N}) - \eta_k \big(\nabla j^k_u(u^k)  - \nabla J_u(u^{*, N}) \big) \|_2^2 \right] + (1 + \f{1}{\epsilon}) \eta_k^2 \E^k[\|\nabla J_u(u^{*, N})  \|^2_2]
\\ \leq & \ (1 + \epsilon)\Big(  \| (u^k - u^{*, N})\|_2^2 - 2 \eta_k  \langle \E^k\big[ \nabla j^k_u(u^k)  - \nabla J_u(u^{*, N})\big] , u^k - u^{*, N} \rangle \\
& \quad + \eta_k^2 \E^k\big[ \| \nabla j^k_u(u^k)  - \nabla J^N_u (u^k) + \nabla J^N_u (u^k) - \nabla J_u(u^{*, N})\|_2^2  \big] \Big) \\
& \qquad + (1 + \f{1}{\epsilon}) \eta_k^2 \E^k[\|\nabla J_u(u^{*, N})  \|^2_2].
\end{aligned}
\end{equation}
From the convexity assumption \eqref{Convexity_Assumption} and Lemma \ref{lemma-convex}, we can deduce by using the Young's inequality with $\lambda/2$ and the fact  that $u^k$ is $\mathcal{G}_k$ measurable to get
\begin{equation}\label{Breakdown-1}
\begin{aligned}
& \ - \langle \E^k\big[ \nabla j^k_u(u^k)  - \nabla J_u(u^{*, N})\big] , u^k - u^{*, N} \rangle = \ - \langle  \nabla J^N_u(u^k)  - \nabla J_u(u^{*, N}), u^k - u^{*, N} \rangle \\
= & \  - \langle  \nabla J^N_u(u^k)  - \nabla J_u(u^{k}), u^k - u^{*, N} \rangle -  \langle \nabla J_u(u^{k})  - \nabla J_u(u^{*, N}), u^k - u^{*, N} \rangle \\
\leq & \f{2 \| \nabla J^N_u(t, u^k)  - \nabla J_u(u^{k}) \|_2^2}{\lambda} + \f{\lambda}{2} \|u^k - u^{*, N}\|_2^2 - \lambda \|u^k - u^{*, N}\|_2^2 \\
\leq &  \f{2}{\lambda}\f{C}{N} - \f{\lambda}{2} \|u^k - u^{*, N}\|_2^2.
\end{aligned}
\end{equation}
Moreover, from Lemma \ref{Error:Gradient}, Lemma \ref{lemma-convex} and the convexity assumption \eqref{Convexity}, we have 
\begin{equation}\label{Breakdown-2}
\begin{aligned}
& \ \E^k\big[ \| \nabla j^k_u(u^k)  - \nabla J^N_u (u^k) + \nabla J^N_u (u^k) - \nabla J_u(u^{*, N})\|_2^2  \big] \\
\leq & \ 2 \E^k\big[ \| \nabla J^N_u (u^k) - \nabla J_u(u^{*, N})\|_2^2  \big]  + C N \\
\leq & \ 4 \Big( \E^k\big[ \| \nabla J^N_u (u^k) - \nabla J_u(u^{k})\|_2^2  \big]  + \E^k\big[ \| \nabla J_u(u^{k}) - \nabla J_u(u^{*, N})\|_2^2  \big]  \Big) + C N \\
\leq & \ 4 \Big( \f{C}{N} + C_L \|u^k - u^{*, N} \|_2^2 \Big) + C N,
 \end{aligned}
\end{equation}
where we use $C$ to denote a generic constant independent of $k$, $N$ and controls. 

Inserting Eqs. \eqref{Breakdown-1}-\eqref{Breakdown-2} in  \eqref{Analysis-convex} and applying  \eqref{J'<C/N}, we obtain
\begin{equation}\label{Analysis-convex-update}
\begin{aligned}
& \E^k\big[ \| u^{k+1} - u^{*, N} \|_2^2 \big] \\
\leq & \ (1 + \epsilon)\Big(  \| (u^k - u^{*, N})\|_2^2 + \f{4}{\lambda} \eta_k  \f{C}{N} - \lambda \eta_k \|u^k - u^{*, N}\|_2^2 \\
& \ \ + 4 \eta_k^2 \big( \f{C}{N} + C_L \|u^k - u^{*, N} \|_2^2 \big) + 4 \eta_k^2 C N  \Big) + (1 + \f{1}{\epsilon}) \eta_k^2 \f{C}{N} \\
= & \ (1 + \epsilon) \Big((1 - c_k\eta_k) \| u^k - u^{*, N} \|_2^2 + 4 \eta_k^2  C N + (\f{4}{\lambda} \eta_k + 4 \eta_k^2 ) \f{C}{N}\Big) + (1 + \f{1}{\epsilon}) \eta_k^2 \f{C}{N},
\end{aligned}
\end{equation}
where $c_k := \lambda - 4 C_L \eta_k$.

Let $\tilde{\eta}_k = \f{1}{k + M}$. We can find $\theta$ and $M$ such that 
$$2 c : = \lambda \theta - 4 C_L \f{\theta^2}{1 + M} > 2,$$
and we have that when $k$ is large enough, $2 c\tilde{\eta}_k \geq c_k \eta_k$  for  $\eta_k = \f{\theta}{k+M}$. By picking $\epsilon = c \tilde{\eta}_k$ in  \eqref{Analysis-convex-update}, we have
\begin{equation*}
\begin{aligned}
& \E^k\big[ \| u^{k+1} - u^{*, N} \|_2^2 \big] \\
\leq & \ (1 + c \tilde{\eta}_k) \Big((1 - 2 c \tilde{\eta}_k)\| u^k_t - u^{*, N} \|_2^2 + C \big(  \tilde{\eta}^2_k N + \f{\tilde{\eta}_k}{N} \big)\Big) + (1 + \f{1}{c \tilde{\eta}_k}) \tilde{\eta}_k^2 \f{C}{N}\\
\leq & \ (1 - c \tilde{\eta}_k) \| u^k_t - u^{*, N} \|_2^2 + C \tilde{\eta}_k^2 N + C \f{\tilde{\eta}_k}{N}.
\end{aligned}
\end{equation*}
Next, we take expectation $\E[\cdot]$ to both sides of the above estimate and apply it recursively from $k=0$ to $k = K$ to get
\begin{equation*}
\begin{aligned}
\E\big[ \| u^{K+1} - u^{*, N} \|_2^2 \big] \leq & \ \prod_{k=0}^{K} (1 - c \tilde{\eta}_k)\E\big[ \| u^{0} - u^{*, N} \|_2^2 \big] + \left( \sum_{m=1}^{K} \tilde{\eta}_{m-1} \prod_{k=m}^{K} (1 - c \tilde{\eta}_k) \right) \f{C}{N} \\
& \quad + \f{C \tilde{\eta}_k}{N} + \left( \sum_{m=1}^{K} \tilde{\eta}^2_{m-1} \prod_{k=m}^{K} (1 - c \tilde{\eta}_k) \right) C N \\
\leq & (K + M)^{-c} \| u^{0} - u^{*, N} \|_2^2 + CN \left( (K+M)^{-1} - \f{(1 + M)^{c-1}}{(K+M)^c} \right) + \f{C}{N}.
\end{aligned}
\end{equation*}
Since $c>1$ and $\prod_{k=m}^{K} (1 - c \tilde{\eta}_k) \sim O\big((K/m)^{-c}\big)$, the above estimate gives us
$$\E\big[ \| u^{K+1} - u^{*, N} \|_2^2 \big] \leq \f{CN}{K+M} + \f{C}{N}\leq  \ds C \left(\f{N}{K} + \f{1}{N}\right).$$

\end{proof}

Note that the estimate  \eqref{Convex:convergence} provides the convergence between the estimated control $u^{K+1}$ and the optimal control found in the subspace $\mathcal{U}_N$. The next theorem, which is the main result of this section,  gives the convergence between $u^{K+1}$ and the exact optimal control $u^* \in \mathcal{U}$.
\begin{thm}\label{thm:Convergence:Convex}
Assume that all the assumptions hold in Theorem \ref{Thm: convex_convergence} and assume the optimal control $u^*$ is bounded. Then, for large enough $K$, we have the following convergence result:
\begin{equation}\label{Convex:Final}
\E\big[ \| u^{K+1} - u^{*} \|_2^2 \big] \leq C \left( \f{N}{K} + \f{1}{N}\right).
\end{equation}
\end{thm} 
\begin{proof}
From the estimate \eqref{Convex:convergence} obtained in Theorem \ref{Thm: convex_convergence} and the fact \eqref{u^N-u^*}, we have
$$
\begin{aligned}
\E\big[ \| u^{K+1} - u^{*} \|_2^2  \big] \leq & \ 2 \E\big[ \| u^{K+1} - u^{*, N} \|_2^2  \big] + 2 \E\big[ \| u^{*, N} - u^{*} \|_2^2  \big]\\
\leq & \ C \left( \f{N}{K} + \f{1}{N}\right) + \f{C}{N}, 
\end{aligned}
$$
as desired.
\end{proof}

\subsection{Convergence analysis for the non-convex optimization case}\label{Section: Non-Convex}
Now, we study the convergence of our algorithm without the convexity assumption. In this work, we introduce the running cost in the stochastic optimal control problem to the loss in machine learning, and the running cost term performs as a regularizer to stabilize the SNN in our analysis. 

In addition to the boundedness assumptions in Assumption \ref{Assumptions}, we assume that the running cost only depends on control $u$, and we let $R(u) : = \int_{0}^T r(u_t) dt $, which is uniformly bounded from below by $C \|u\|_2$, i.e. 
\begin{equation}\label{Bounded:R}
R^2(u) \geq C \|u\|^2_2,  
\end{equation}
where $C$ is a constant that satisfies Lemma \ref{lemma-convex}. We also assume the learning rate $\eta_k$ satisfies
\begin{equation}\label{ass:eta}
\sum_{k=1}^{\infty} \eta_k = \infty, \qquad \sum_{k=1}^{\infty} (\eta_k)^2 < \infty.
\end{equation}


To proceed, we denote $J^N$ as the cost function corresponding to the fully calculated approximate gradient $\nabla J_u^N$, i.e., 
$$\lim_{\delta \rightarrow 0} \f{J^N(u + \delta v) - J^N(u)}{\delta} = \langle\nabla J^N_u(u), v \rangle, \qquad u, v \in \mathcal{U}_N,$$
Note  that for any $u_0$, $u \in \mathcal{U}_N$,
\begin{equation}\label{Prop:J^N}
\begin{aligned}
J^N(u) = & \ J^N(u_0) + \int_0^1 \f{d}{d \epsilon} \Big( J^N\big(u_0 + \epsilon (u - u_0) \big) \Big) d\epsilon \\
=& \ J^N(u_0) + \int_0^1  \big\langle \nabla J_u^N\big(u_0 + \epsilon (u - u_0) \big),  u - u_0 \big\rangle d\epsilon.
\end{aligned}
\end{equation}
Then, we can show that $J^N$ is bounded from below based on the boundedness assumption for the running cost $R$, i.e.,
\begin{equation*}\label{J^N:bounded}
\begin{aligned}
J^N(u^k) = & \ J(u^k) - J(u_0) + J^N(u_0) - \int_{0}^{1} \big\langle \nabla J_u\big(u_0 + \epsilon (u^k - u_0)\big) \\
 & \qquad -  \nabla J^N_u\big(u_0 + \epsilon (u^k - u_0)\big), (u^k - u_0) \big\rangle d\epsilon \\
\geq &  \ J(u^k) - C_0 - C \sup_{u^k \in \mathcal{U}_N}\|\nabla J_u(u^k) - \nabla J^N_u(u^k)\|_2  \|u^k - u_0\|_2 \\
= & \ \Phi(X^k_T) - C_0 + R(u^k) - C \sup_{u^k \in \mathcal{U}_N}\|\nabla J_u(u^k) - \nabla J^N_u(u^k)\|_2  \|u^k - u_0\|_2,
\end{aligned}
\end{equation*}
where $X^k_T$ is the solution of the state equation \eqref{SNN-continuous} driven by the control $u^k$.
We choose $u_0 = 0$ in the above estimate. Then, it follows from Lemma \ref{lemma-convex} that 
$$J^N(u^k)  \geq  \Phi(X^k_T) - C_0 + \underbrace{R(u^k) - \sqrt{C/N} \|u^k\|_2}_{ \mathlarger{\geq 0} } \geq \Phi(X^k_T) - C_0,$$
which indicates that $J^N$ is bounded from below.

Our aim is to  show that for a given depth $N$, $\lim
_{k \rightarrow \infty} \|\nabla J_u(u^k)\|_2 \rightarrow 0$ a.s.. To this end, we first state  two propositions that can be derived following the same  proofs as  in Theorem 5.3.1 and Theorem 5.3.3 of \cite{Zhang_BSDE}.
\begin{prop}\label{Prop:X_uv}
Assume the assumptions (a), (b), (d) in Assumption \ref{Assumptions} hold. Let $X_n^{N, u}$ and $X_n^{N, v}$ be approximate solutions introduced in  \eqref{scheme:X} driven by two controls $u, v \in \mathcal{U}_N$, and we denote $\Delta_N X^{u, v}_n : = X_n^{N, u} - X^{N, v}_n$. Then, we have
\begin{equation*}
\max_{0\leq n \leq N}|\Delta_N X_n^{u, v} |^2 \leq C | u - v|^2.
\end{equation*}
\end{prop}
\begin{prop}\label{Prop:YZ_uv}
Assume that Assumption \ref{Assumptions} hold and  let $\Delta_N Y^{u, v}_n : = Y_n^{N, u} - Y^{N, v}_n$ and $\Delta_N Z^{u, v}_n : = Z_n^{N, u} - Z^{N, v}_n$ for numerical solutions $Y^N$ and $Z^N$ with controls $u, v \in \mathcal{U}_N$. We have the following estimate
\begin{equation*}
\sum_{0\leq n \leq N} \E[|\Delta_N Y^{u, v}|^2] + h \sum_{n=0}^{N-1}\E[|\Delta_N Z^{u, v}_n|^2] \leq C |u - v|^2.
\end{equation*}
\end{prop}
With the above propositions,  we can derive the following lemma:
\begin{lem}\label{Lip:J'}
Let $w, v \in \mathcal{U}_N$. Under Assumption \ref{Assumptions}, there exists a positive constant $C$ such that 
$$\|\nabla J^N_u(w) - \nabla J^N_u(v)\|_2^2 \leq C \|w - v\|_2^2.$$
\end{lem}
\begin{proof}
Since $w, v \in \mathcal{U}_N$, we can write $w = (w_0, w_1, \cdots, w_n, \cdots, w_N)$ and  $v = (v_0, v_1, \cdots, v_n, \cdots, v_N)$. Then, it follows from Proposition \ref{Prop:X_uv},  Proposition \ref{Prop:YZ_uv} and Assumption \ref{Assumptions} that
$$
\begin{aligned}
& \|\nabla J^N_u(u) - \nabla J^N_u(v)\|_2^2 \\
\leq & \ h \sum_{n=0}^{N-1} \Big( \E\big[ | f_u(X_n^{N, w}, w_n)^\top Y_n^{N, w} - f_u(X_n^{N, v}, v_n)^\top Y_n^{N, v} |  \\
&  \quad +  | g_u(X_n^{N, w}, w_n)^\top Z_n^{N, w} - g_u(X_n^{N, v}, v_n)^\top Z_n^{N, v} | + |r_u(w_n)^\top - r_u(v_n)^\top| \big] \Big)^2 \\
\leq & \ C h \sum_{n=0}^{N-1} \E\big[ |Y_n^{N, w} - Y_n^{N, v}|^2 + |Z_n^{N, w} - Z_n^{N, v}|^2 + |X_n^{N, w} - X_n^{N, v}|^2 + |w_n - v_n	|^2\big] \\
\leq & C h \sum_{n=0}^{N-1}|w_n - v_n|^2. 
\end{aligned}
$$

\end{proof}

We also need   the following lemma.
\begin{lem}\label{Diff:J^N}
Let $\{u^k\}_k$ be the sequence of approximate  controls obtained by \eqref{SGD-u}. Under Assumption \ref{Assumptions} and the bounded from below assumption \eqref{Bounded:R}, we have the following estimate
\begin{equation}\label{Ineq:Diff:J^N}
\E^k[J^N(u^{k+1})] \leq J^N(u^k) - \eta_k \|\nabla J^N_u(u^k)\|_2^2 + C N \eta_k^2.
\end{equation}
\end{lem}
\begin{proof}
From the definition of $J^N$ and  \eqref{Prop:J^N}, we have that
$$
\begin{aligned}
& J^N(u^{k+1}) - J^N(u^k) - \langle\nabla J^N_u(u^{k}), u^{k+1} - u^k \rangle \\
= & \ \int_0^1  \big\langle \nabla J_u^N\big(u^k + \epsilon (u^{k+1} - u^k) \big) - \nabla J^N_u(u^k), \  u^{k+1} - u^k \big\rangle d\epsilon.
\end{aligned}
$$
Applying Lemma \ref{Lip:J'} to the right hand side of the above equation, we get
$$J^N(u^{k+1}) - J^N(u^k) - \langle\nabla J^N_u(u^{k}), u^{k+1} - u^k \rangle  \leq C \|u^{k+1} - u^k\|_2^2.$$
Hence we can rewrite the above inequality to get the following estimate 
\begin{equation}\label{Est:J^N}
J^N(u^{k+1}) \leq J^N(u^k) - \eta_k \langle \nabla J^N_u(u^{k}), \nabla j^k_u(u^k)\rangle + C (\eta_k)^2\| \nabla j^k_u(u^k)\|_2^2.
\end{equation}
Next, we take the conditional expectation $\E^k[\cdot]$ on both sides of  \eqref{Est:J^N}. Similar to the argument  in proving Proposition \ref{Y^k=Y^N}, we have 
$$ \E^k[\langle \nabla J^N_u(u^{k}), \nabla j^k_u(u^k)\rangle] = \|\nabla J^N_u(u^k)\|_2^2,$$
which, together with the estimate estimate \eqref{bound:j'}, gives us
$$\E^k[J^N(u^{k+1})] \leq J^N(u^k) - \eta_k \|\nabla J^N_u(u^k)\|_2^2 + C N \eta_k^2 $$
as desired. 
\end{proof}

The following lemma gives the final preparation for the main theorem of this subsection.
\begin{lem}\label{Lem:Final}
Under Assumption \ref{Assumptions} and the bounded from below assumption \eqref{Bounded:R}, suppose that 
\begin{equation}\label{ass:bound:sum}
\sum_{k=0}^{\infty} \eta_k \E[\|\nabla J^N_u(u^k)\|_2^2] < \infty.
\end{equation}
Then we have \quad
$$ \lim_{k\rightarrow \infty}\|\nabla J^N_u(u^k)\|_2^2 = 0, \qquad a.s. .$$
\end{lem}
\begin{proof}
We first show that
\begin{equation}\label{lim-inf}
\liminf_{k\rightarrow \infty} \|\nabla J^N_u(u^k)\|_2^2 = 0, \qquad a.s. .
\end{equation}
If  \eqref{lim-inf} is not true, there exists constants $K>0$ and $a > 0$ such that for all $k > K$, $\|\nabla J^N_u(u^k)\|_2^2 > a$. As a result, we have from our assumption in  \eqref{ass:eta} that
$$\sum_{k=K+1}^{\infty}  \eta_k \|\nabla J^N_u(u^k)\|_2^2 >  a \sum_{k=K+1}^{\infty} \eta_k = \infty, $$
which contradicts the assumption \eqref{ass:bound:sum} in this lemma. 

On the other hand, we suppose that
\begin{equation}\label{lim-sup}
\limsup_{k\rightarrow \infty} \|\nabla J^N_u(u^k)\|_2^2 > 0, \qquad a.s. .
\end{equation}
Then, we can find two sequences of stopping times $\{m_k\}_k$ and $\{n_k\}_k$ defined inductively as follows: given $\epsilon > 0$, let
$$m_0:=\inf\{k: \|\nabla J^N_u(u^k)\|_2^2 > 2 \epsilon\},\vspace{-0.5em}$$ 
$$n_k:=\inf\{k>m_k: \|\nabla J^N_u(u^k)\|_2^2 < \epsilon\},\vspace{-0.5em}$$
$$m_{k+1}:=\inf\{k>n_k: \|\nabla J^N_u(u^k)\|_2^2 > 2 \epsilon\}.$$
Hence, we have
$$\infty > \sum_{k=0}^{\infty} \eta_k \|\nabla J^N_u(u^k)\|_2^2 \geq \sum_{k=0}^{\infty} \sum_{i=m_k}^{n_k-1}\eta_i \|\nabla J^N_u(u^i)\|_2^2 \geq \epsilon \sum_{k=0}^{\infty} \sum_{i=m_k}^{n_k-1}\eta_i. $$
Therefore, $\lim_{k\rightarrow \infty}\sum_{i=m_k}^{n_k-1}\eta_i = 0$, $a.s.$. By  \eqref{bound:j'} in the proof of Lemma \ref{Error:Gradient}, we have
$$\E^k[\|u^{k+1} - u^k\|_2^2]  = \eta_k^2 \E^k[\|\nabla j^k_u(u^k)\|_2^2] \leq C N \eta_k^2.$$
By triangle inequality, the above estimate gives 
$$\|u^{n_k} - u^{m_k}\|_2 \leq \sqrt{CN} \sum_{i=m_k}^{n_k-1} \eta_i \longrightarrow 0, \ \ \text {as } k\rightarrow \infty.$$
Then, by Lemma \ref{Lip:J'} we obtain 
\begin{equation}\label{lim_J'=0}
\lim_{k\rightarrow \infty} \|\nabla J^N_u(u^{n_k}) - \nabla J^N_u(u^{m_k} )\|_2^2 \longrightarrow 0, \quad a.s. .
\end{equation}
By definition of stopping times $\{m_k\}_k$ and $\{n_k\}_k$, we have
$$
\begin{aligned}
\epsilon < & \|\nabla J^N_u(u^{m_k})\|_2^2 - \|\nabla J^N_u(u^{n_k})\|_2^2 = \|\nabla J^N_u(u^{m_k}) - \nabla J^N_u(u^{n_k}) + \nabla J^N_u(u^{n_k})\|_2^2 - \|\nabla J^N_u(u^{n_k})\|_2^2 \\
\leq & \ (1 + 2) \|\nabla J^N_u(u^{m_k}) - \nabla J^N_u(u^{n_k})\|_2^2 + (1+\f{1}{2}) \|\nabla J^N_u(u^{n_k})\|_2^2 - \|\nabla J^N_u(u^{n_k})\|_2^2 \\
\leq &\ 3 \|\nabla J^N_u(u^{m_k}) - \nabla J^N_u(u^{n_k})\|_2^2 + \f{1}{2}  \|\nabla J^N_u(u^{n_k})\|_2^2 \\
< & \ 3 \|\nabla J^N_u(u^{m_k}) - \nabla J^N_u(u^{n_k})\|_2^2  + \f{1}{2}\epsilon.
\end{aligned}
$$
However, letting $k\rightarrow \infty$ in the above inequality will result a contradiction due to  \eqref{lim_J'=0}.

Therefore, we have $\limsup_{k\rightarrow \infty} \|\nabla J^N_u(u^k)\|_2^2 = 0$, $a.s.$. Together with  \eqref{lim-inf}, we obtain the desired convergence result in the lemma. 
\end{proof}

Now, we introduce our main convergence theorem that shows our sample-wise back-propagation method convergences in training a $N$-layer SNN.
\begin{thm}
Under Assumption \ref{Assumptions} and the bounded from below assumption \eqref{Bounded:R}, we have the following result for a given integer $N \in \mathbb{N}$:
\begin{equation*}
\lim_{k\rightarrow \infty}\|\nabla J^N_u(u^k)\|_2^2 = 0, \qquad a.s. .
\end{equation*}
\end{thm}
\begin{proof}
Let $\beta_k = \eta_k \|\nabla J^N_u(u^k)\|_2^2$. We proceed to prove $\sum_{k=0}^{\infty}\E[\beta_k] < \infty$, which is the condition  \eqref{ass:bound:sum} in Lemma \ref{Lem:Final} that will give the desired result of this theorem.

Define 
$$\lambda_k = J^N(u^k) + CN \sum_{i=k}^{\infty} \eta_i^2,$$
where $C$ is a constant that satisfies  \eqref{Ineq:Diff:J^N} in Lemma \ref{Diff:J^N}.
Then, we apply  \eqref{Ineq:Diff:J^N} to get
\begin{equation}\label{lambda:inequality}
\E^k[\lambda_{k+1}] = \E^k[J^N(u^{k+1})] + CN \sum_{i=k+1}^{\infty} \eta_i^2 \leq J^N(u^k) + CN \sum_{i=k}^{\infty} \eta_i^2 - \beta_k = \lambda_k - \beta_k < \lambda_k.
\end{equation}
Since $J^N$ is bounded from below as we discussed above, $\E[\lambda_k]$ is also bounded from below.
Moreover, we know from  \eqref{lambda:inequality} that $\{\lambda_k\}_k$ is a super-martingale bounded from below. Therefore, by martingale convergence theorem, we obtain
$$\lim_{k \rightarrow \infty} \E[\lambda_k] < \infty.$$
Then, we apply  \eqref{lambda:inequality} to get
$$
\begin{aligned}
\sum_{k=0}^{\infty}\E[\beta_k] \leq & \ \sum_{k=0}^{\infty}\E\left[\lambda_k - \E^k[\lambda_{k+1}]\right] \\
= & \  \sum_{k=0}^{\infty}\big( \E[\lambda_k] - \E[\lambda_{k+1}]\big)  < \infty.
\end{aligned}
$$
Then, the desired result of this theorem holds by using the conclusion in Lemma \ref{Lem:Final}.
\end{proof}


\section{Numerical experiments} \label{Numerics}

In this section, we use three numerical examples to validate our theoretical results and demonstrate the performance of SNN (as a probabilistic learning method) trained by our sample-wise back-propagation algorithm. In Example 1, we solve a linear-quadratic stochastic optimal control problem, which satisfies the convexity assumption, and the analytical expression of the optimal control can be found in this example. We want to use this classic stochastic optimal control example to examine the convergence rate that we provided in Theorem \ref{thm:Convergence:Convex}. In Example 2, we present the performance of our sample-wise back-propagation method in training SNNs to learn random noise perturbed functions, and we will show the convergence trend of our algorithm in this function approximation example. We will first use a smaller scale SNN to learn a one-dimensional (1D) function, and then let SNN learn an 8D function. In the third example, we use our SNN method to solve the benchmark classification problem by using the MNIST handwritten dataset and the Fashion-MNIST dataset. We want to use this example to show that our method works well for classic machine learning tasks, and it also has the robustness advantage compared with the conventional deterministic convolution neural network. 
The CPU that we use to run the numerical experiments in this section is an M1 Pro Chip with 3.2GHz and 16 GB memory.

\subsection{Example 1}

In this example, we consider the following state process on the temporal interval $[0, T]$:
\begin{equation}\label{Ex1:state}
dX_t = (u_t - a_t) dt + \sigma u_t dW_t,
\end{equation}
where $X$ and $u$ are 8D state and control. The vector function $a_t$ is defined by
$$
\begin{aligned}
a_t = \left[\f{-t^2}{2\beta_t}, \f{-\sin t}{\beta_t}, \f{-0.5 \exp(1 -t)}{\beta_t}, \f{-t^3}{3 \beta_t}, \f{-\ln(1+t)}{\beta_t}, \f{-\cos 2\pi t}{\beta_t}, \f{-\tan t}{\beta_t}, \f{1-t}{2 \sigma^2 (1-t) + 2}\right]^\top,
\end{aligned}
$$
where $\beta_t = (1 + \sigma^2) + \sigma^2(1 - t)$.
The cost functional is defined by
\begin{equation}\label{Ex1:cost}
J(u) = \f{1}{2}\int_0^1 \E[| X_t - X^*_t |^2] dt + \f{1}{2}\int_0^1 |u_t|^2 dt + \f{1}{2}|X_T|^2,
\end{equation}
where $|\cdot|$ is the Euclidean norm in $\mathbb{R}^8$, and $X^*_t$ is defined by
{\small$$
\begin{aligned}
& X^*_t :=\Big[ t + \alpha_t \f{0.5 - x_T^{(1)}}{\sigma^2}, \cos t + \alpha_t \f{\sin 1 - x_T^{(2)}}{\sigma^2}, -\f{\exp(1 - t)}{2} + \alpha_t\f{0.5 - x_T^{(3)}}{\sigma^2}, t^2 + \alpha_t \f{1/3 -x_T^{(4)}}{\sigma^2}, \\
& \quad \f{1}{1+t} + \alpha_t \f{\ln 2 - x_T^{(5)}}{\sigma^2}, - 2 \pi \sin t+ \alpha_t \f{2 \pi \cos 1 - x_T^{(6)}}{\sigma^2}, \sec^2 t + \alpha_t \f{\tan 1 - x_T^{(7)}}{\sigma^2}, \f{t}{\sigma^2} + 1 - \f{1}{2 \sigma^4} \ln \f{1 + \sigma^2}{1 + \sigma^2 (1 - t)} \Big]^\top,
\end{aligned}
$$}
where $\ds \alpha_t = \ln \f{1 + 2 \sigma^2}{\sigma^2 (2 - t) + 1}$. For $\ds D : = \f{\ln (1 + \f{\sigma^2}{1 + \sigma^2})}{\sigma^2 + \ln (1 + \f{\sigma^2}{1 + \sigma^2})}$, $x_T = [x_T^{(1)}, \cdots, x_T^{(7)}]$ in the above are defined as $x_T  : = \Big[ D/2, D \cdot \sin 1, D/2, D, D\cdot  \ln 2,  D\cos2 \pi,  D \cdot \tan 1 \Big]^\top$. Since the stochastic optimal control problem \eqref{Ex1:state}-\eqref{Ex1:cost} is a linear-quadratic problem, we can find the analytic expression for the optimal control as
{\small$$
\begin{aligned}
& u^*_t :=\Big[ \f{-t^2/2 + T^2/2 - x_T^{(1)}}{\beta_t}, \f{-\sin t + \sin 1 - x_T^{(2)}}{\beta_t},  \f{-1/2 \exp(T - t) + 1/2 - x_T^{(3)}}{\beta_t}, \f{-t^3 + T^3 - 3x_T^{(4)}}{3 \beta_t}, \\
& \quad \f{-\ln (1+t) + \ln(1 + T) - x_T^{(5)}}{\beta_t},   \f{- \cos 2\pi t + \cos 2 \pi - x_T^{(6)}}{\beta_t},\f{- \tan t + \tan 1 - x_T^{(7)}}{\beta_t}, \f{T - t}{\sigma^2 (T - t) + 1 } \Big]^\top.
\end{aligned}
$$}

In this example, we let $T=1$, $\sigma = 0.5$, and $X_0 = 0$. In the first experiment, we let $N = 20, 30, 40, \cdots, 100$ with iteration steps $K = 0.2 \times N^2$ for each $N$. In the convergence analysis  \eqref{Convex:Final} proved in Theorem \ref{thm:Convergence:Convex}, we can see that when choosing $K \sim O(N^2)$, the approximation error $\|u^{K+1}-u^{*}\|_2$ has half-order convergence rate. To validate this result, we solve the stochastic optimal control problem \eqref{Ex1:state}-\eqref{Ex1:cost} repeatedly $50$ times and plot the root mean square errors (RMSEs) in Figure \ref{Ex1_Convergence} (a).
\begin{figure}[!htb]
\center
   \subfloat[Convergence with respect to $N$ ($K = 0.2 N^2$).]{ \includegraphics[width=0.42\textwidth]{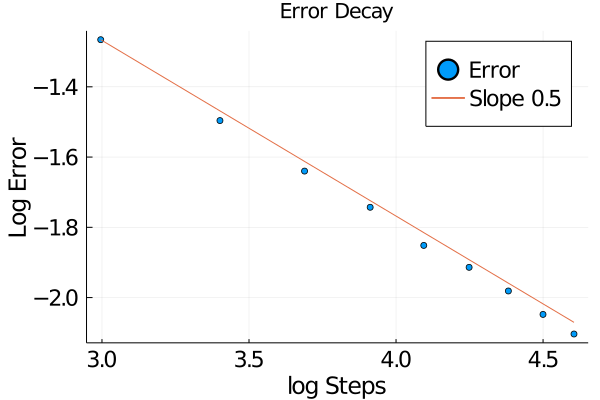}} \hspace{1em}
   \subfloat[Convergence with respect to $K$.]{ \includegraphics[width=0.42\textwidth]{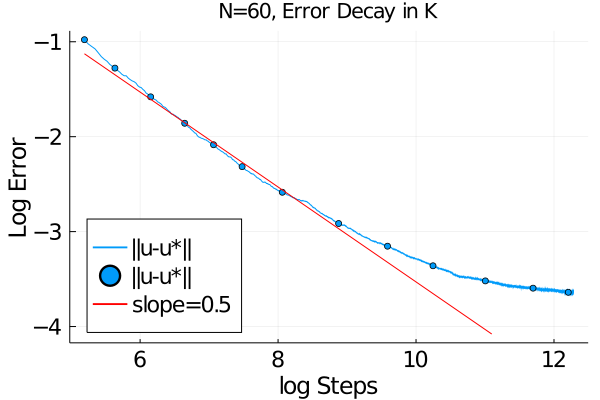}}
    \caption{Example 1: Convergence results.}
    \label{Ex1_Convergence}
\end{figure}
The blue dots give the RMSEs corresponding to $N$ (presented by $\log N$ on the $x$-axis), and the red straight line shows the slope of half-order convergence. From this figure, we can see that our algorithm does provide half-order convergence rate in this example when the relation $K \sim O(N^2)$ is satisfied. 

In Figure \ref{Ex1_Convergence} (b), we further investigate the convergence with respect to $K$, and we let $N= 60$ be a fixed partition number. The blue dots show the RMSEs corresponding to different $K$ values (presented by $\log K$ on the $x$-axis), and the red straight line indicates the slope of half-order convergence. From this figure, we can see that the convergence rate of our algorithm is well-aligned with the half-order slope while the number of iterations $K$ is relatively small. However, when $K$ values become larger, the accuracy of our algorithm does not always improve accordingly. This verifies the existence of the $\f{C}{N}$ term in the analysis  \eqref{Convex:Final}, which becomes the bottleneck in convergence with large iteration number $K$. In other words, a fixed partition number $N$ limits the convergence of algorithm no matter how many iteration steps we implement in the optimization procedure.  


\subsection{Example 2}
In this example, we use the SNN model \eqref{SNN-discrete} to learn random noise perturbed functions by using discrete function values as data and demonstrate the performance of our method in both function approximation and uncertainty quantification. We let $h=1$ and choose the activation function in the SNN as a bounded combination of sigmoid functions, which is defined as follows:
$$b(X) = \sum_{l=1}^{L} a_l \sigma(W X + V),$$
where $\sigma(x) = \f{1}{1+e^{-x}}$ is a sigmoid function, $\{a_l\}_{l=1}^{L}$ are taking values on the interval $[-4.5, 4.5]$, which is a set of weights for different sigmoid functions, $W$ is the weight matrix for state of neurons, and $V$ is a bias vector.

\subsubsection*{Case 1: 1D function approximation}
In this experiment, we approximate a 1D function: $f(x) = \sin (2 \pi x) + 0.05 \xi$, $\xi \sim N(0,1)$, and we collect altogether $10,000$ noise perturbed function values, i.e. input-output pairs $\{(x_i, f(x_i))\}_{i=1}^{10,000}$, as our data set.

In this experiment, we train an $8$ layer - $4$ neuron SNN and run $K = 8 \times 10^6$ iteration steps in the training procedure. The CPU time for this training procedure is $284.54$ seconds.
\begin{figure}[!htb]
\center
\subfloat[SNN approximation with bounded activation]{\includegraphics[scale = 0.35]{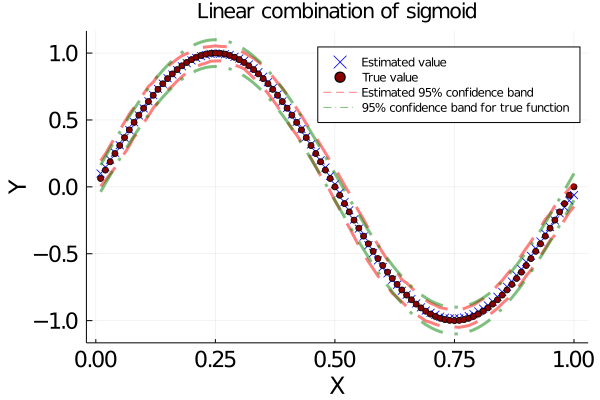} } \vspace{0.5em}
\subfloat[SNN approximation with unbounded activation]{\includegraphics[scale = 0.35]{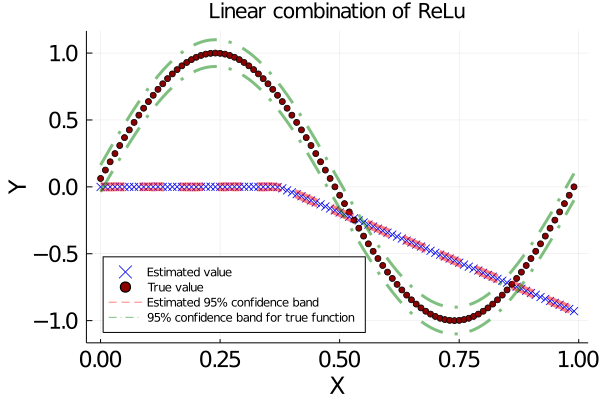} } 
 \caption{Example 1: SNN performance in 1D function approximation.}
    \label{Ex1_Unbounded}
    \vspace{-1em}
\end{figure}
Then, we use the trained SNN to estimate the function. The approximated function is presented in Fig. \ref{Ex1_Unbounded} (a) by showing the mean of SNN outputs (the blue crosses) on a uniform mesh over $[0, 1]$, and the corresponding $95\%$ confidence band (the red dashed lines) is also presented in this figure. We can see that the SNN estimates are well-aligned with the true function mean values (the red dots), and the estimated $95\%$ confidence band is very close to the true function's $95\%$ uncertainty band (the green lines).
Note that in our analysis, we require boundedness and differentiability for the activation function in the SNN. To show that violation of boundedness and differentiability may result poor performance of SNN, we repeat the above experiment by replacing sigmoid function $\sigma$ with a ReLu function. It is well known that ReLu is a very important activation function in deep learning, and it is a typical activation function that is not bounded and is not differentiable everywhere. In this experiment, we also train a $8$ layer - $4$ neuron SNN as we did to obtain Figure \ref{Ex1_Unbounded} (a), and the result is presented in Figure \ref{Ex1_Unbounded} (b).  We can see from this figure that without boundedness and differentiability for the activation function, the SNN fails in this experiment. Since both experiments are carried out under the same environment except that the activation function is chosen as sigmoid (subplot (a)) and ReLu (subplot (b)). This verifies the necessity of boundedness and differentiability assumption in our algorithm.

\subsubsection*{Case 2: 8D function approximation}
In the second experiment, we approximate the following noise perturbed 8D function:
$$
\begin{aligned}
f(x_1, \dots, x_8) := \exp(x_1) &\cos(2 \pi x_2)  + 8 x_3 (x_4 - 0.5)^2 + x_5 \\
& + \log(2 + x_6) + x_7^2 + 2 x_8 + 0.05 \xi, \quad \xi \sim N(0,1).
\end{aligned}
$$
The training data are collected on $6^8$ spatial mesh points in the hyper cube $[0, 1]^8$. Since this is a high dimensional function, we use a SNN with $N=15$ layer and put $40$ neurons in each layer. The number of optimization iterations that we carry out to train the SNN is $1.5 \times 10^7$, and the CPU time for this training procedure is $1508$ seconds. It's worthy to mention that each data sample contains only limited amount of information about the perturbation noise, and significant amount of computational effort in this 8-dimensional approximation example would contribute to finding the confidence band of the function, which is also the main challenge in uncertainty quantification for deep learning. 

\begin{figure}[htb!]
\begin{center}
\subfloat[Exact surface: $X_2$-$X_5$]{\includegraphics[scale = 0.28]{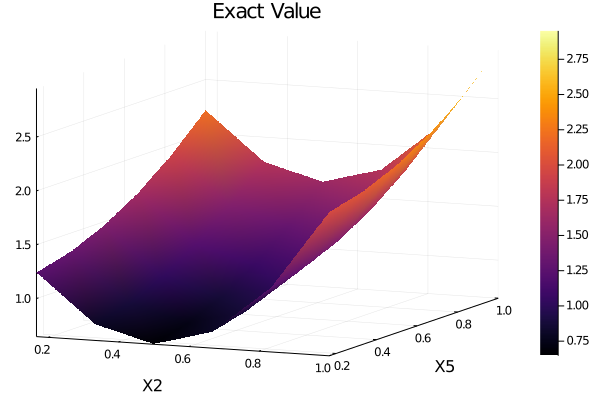} } \vspace{0.5em}
\subfloat[Approximate surface: $X_2$-$X_5$]{\includegraphics[scale = 0.28]{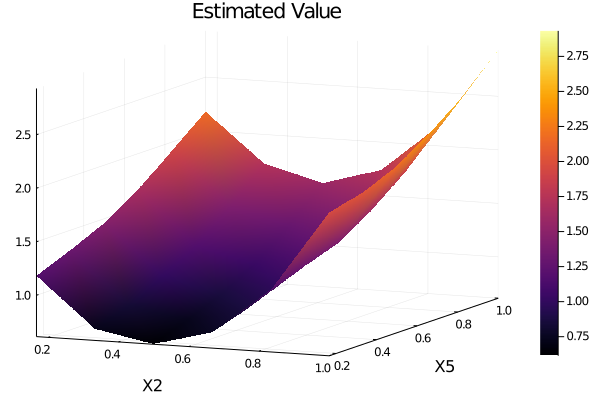} } \\
\subfloat[Exact surface: $X_4$-$X_7$]{\includegraphics[scale = 0.28]{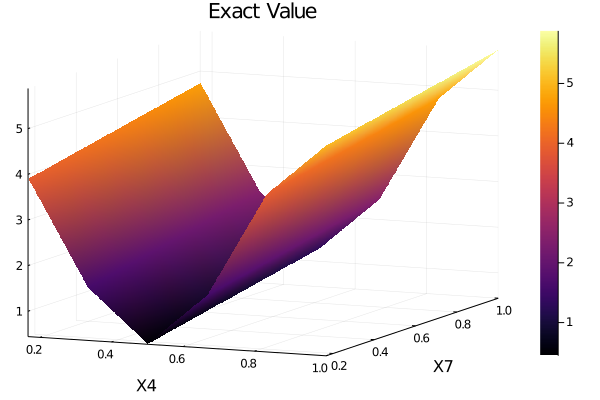} } \vspace{0.5em}
\subfloat[Approximate surface: $X_4$-$X_7$]{\includegraphics[scale = 0.28]{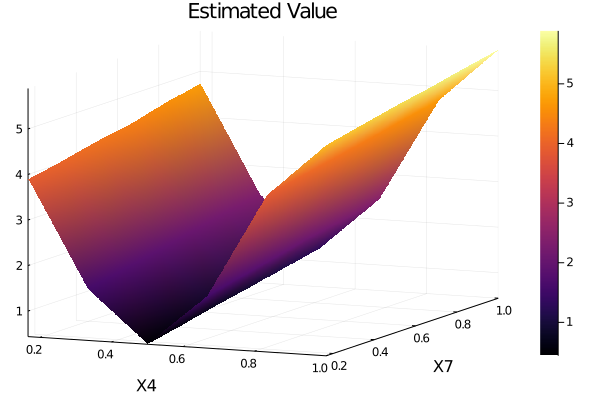} } 
\end{center}
\caption{8D Function approximation -- surface views.}\label{8D_surface}     \vspace{-1em}
\end{figure}
In Figure \ref{8D_surface}, we compare the real marginal surface of function $f$ (presented in mean values) with our estimated surface by using the trained SNN. In Figure \ref{8D_surface} (a), we present the exact $X_2$-$X_5$ marginal surface of function $f$, and Figure \ref{8D_surface} (b) shows the SNN learned $X_2$-$X_5$ marginal surface; In Figure \ref{8D_surface} (c), we present the exact $X_4$-$X_7$ marginal surface of function $f$, and Figure \ref{8D_surface} (d) shows the SNN learned $X_4$-$X_7$ marginal surface. From this figure, we can see that by training the SNN $1.5 \times 10^7$ steps with our sample-wise back-propagation algorithm, it can give very good approximation for the original function.  

\begin{figure}[htb!]
\begin{center}
\subfloat[$X_1$ section]{\includegraphics[scale = 0.28]{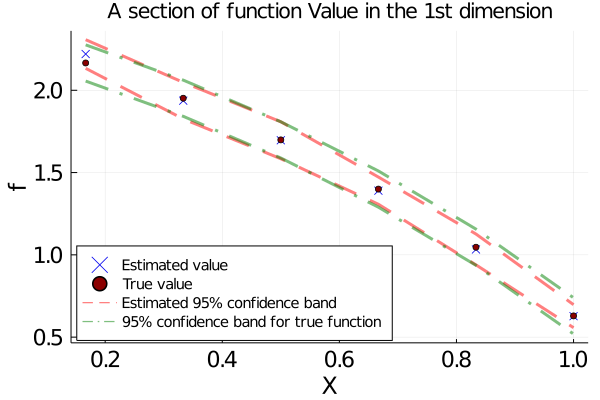} } 
\subfloat[$X_2$ section]{\includegraphics[scale = 0.28]{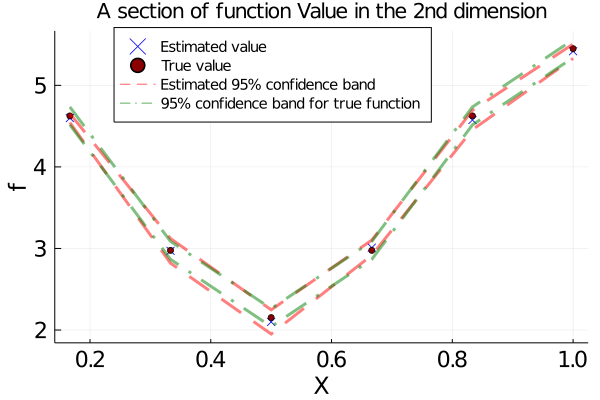} } \\
\subfloat[$X_3$ section]{\includegraphics[scale = 0.28]{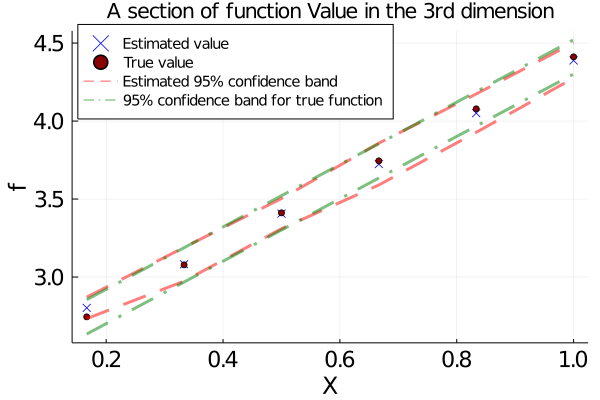} } 
\subfloat[$X_4$ section]{\includegraphics[scale = 0.28]{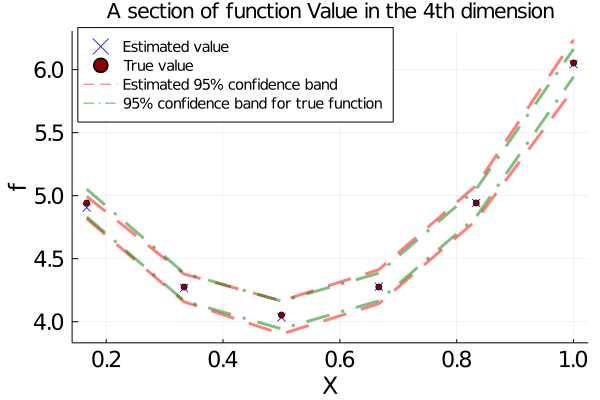} } \\
\subfloat[$X_5$ section]{\includegraphics[scale = 0.28]{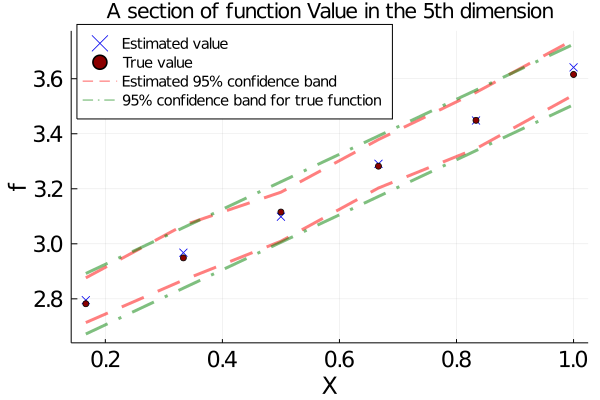} } 
\subfloat[$X_6$ section]{\includegraphics[scale = 0.28]{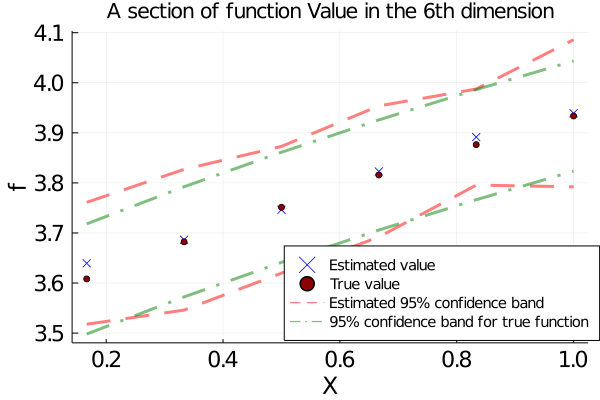} } \\
\subfloat[$X_7$ section]{\includegraphics[scale = 0.28]{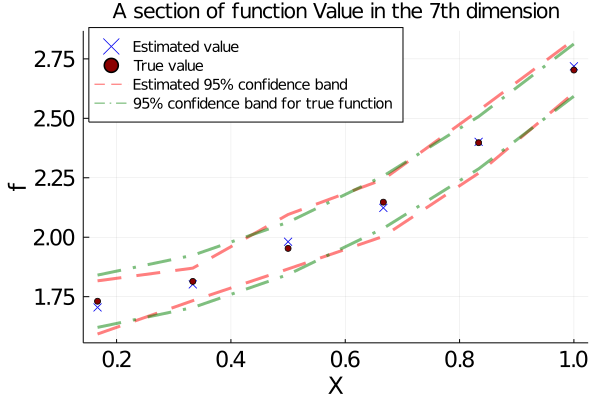} } 
\subfloat[$X_8$ section]{\includegraphics[scale = 0.28]{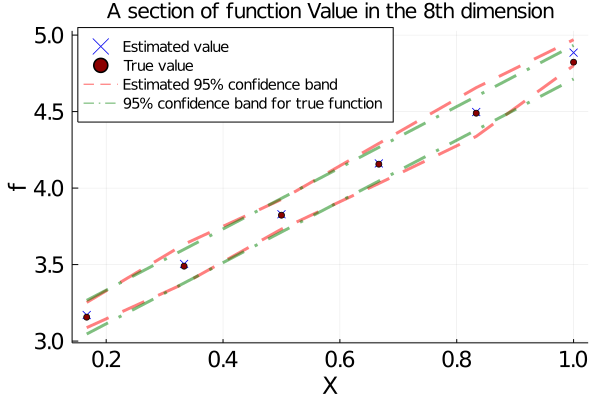} }
\end{center}
\caption{8D Function approximation -- section views}\label{8D_section} 
    \vspace{-1em}
\end{figure}
To show more details of the SNN's performance in function approximation, we present section views of each direction in Figure \ref{8D_section}, where the exact function mean values are plotted by red dots, the SNN estimated function mean values are plotted by blue crosses, the true $95\%$ uncertainty bands are presented by green dashed lines, and the SNN estimated $95\%$ confidence bands are presented by red dashed lines.  From this figure, we can see that the trained SNN can accurately approximate the function's mean values, as well as the uncertainty caused by the random variable $\xi$.


\begin{figure}[!htb]
\center
\includegraphics[scale = 0.4]{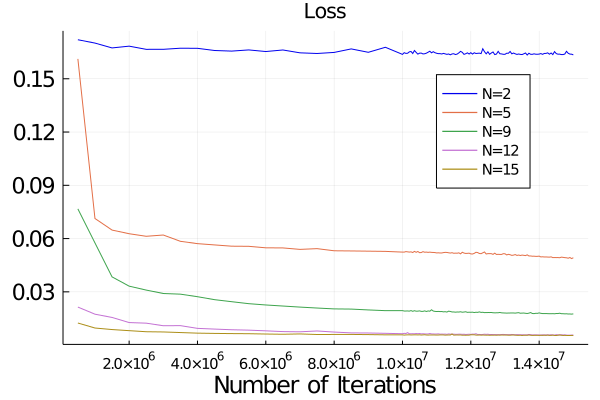} 
\caption{Example 2: Convergence in function approximation.}
    \label{8D_Convergence_K}
        \vspace{-1em}
\end{figure}
To validate the convergence of our algorithm, we train SNNs with $N = 2, 5, 9, 12, 15$, and present the accuracy of function approximation and uncertainty quantification with different numbers of training steps in Figure \ref{8D_Convergence_K} and Figure \ref{8D_Convergence_UQ_K}.
\begin{figure}[!htb]
\center
\includegraphics[scale = 0.4]{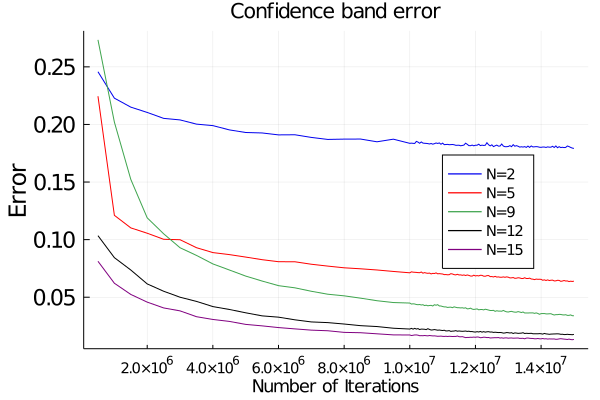} 
\caption{Example 2: Convergence in uncertainty quantification.}
    \label{8D_Convergence_UQ_K}
        \vspace{-1em}
\end{figure}
We can see from those figures that the SNN can achieve higher accuracy in estimating both function values and uncertainty bands by using more layers. For a fixed SNN depth, we can see clear convergence trend when implementing more iteration steps. However, an accuracy barrier appears in each experiment due to the fixed SNN depth, which typically result limited representation capability by using neural network models.

\subsection{Example 3}
In this example, we present the performance of our SNN algorithm in solving benchmark machine learning problems. Specifically, we solve the image classification problem with the MNIST handwritten digit dataset and the Fashion-MNIST dataset. In each dataset, we have 60000 training data samples and 10000 testing data samples. 
To enhance the learning capability of SNN to accomplish the classification task, we incorporate convolution blocks into the drift term in the SNN framework. Other DNN structures can be adopted in different machine learning scenarios. Also, instead of using a single-realization of sample to represent the state variable (as we did in the above synthetic examples), we use a mini-batch of samples to approximate expectations in our SNN algorithm. Detailed formulation of the SNN structure and our specific numerical implementation can be found on github at \textit{https://github.com/Huisun317/SNN}. 

In Figure \ref{MINST}, we present the training accuracy and testing accuracy of our SNN algorithm in solving the classification problem over the MINST handwritten dataset with different batch sizes,
\begin{figure}[htb!]
\begin{center}
\subfloat[Training accuracy]{\includegraphics[scale = 0.8]{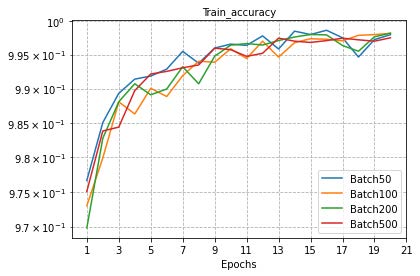} } \vspace{0.5em}
\subfloat[Testing accuracy]{\includegraphics[scale = 0.8]{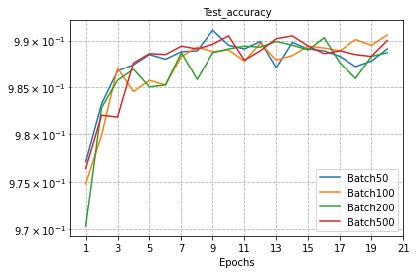} } 
\end{center}
\caption{Performance of SNN: MINST handwritten dataset.}\label{MINST}     \vspace{-1em}
\end{figure}
and in Figure \ref{Fashion-MINST}, we present the training accuracy and testing accuracy of our SNN algorithm in solving the classification problem over the Fashion-MINST dataset with different batch sizes.
\begin{figure}[htb!]
\begin{center}
\subfloat[Training accuracy]{\includegraphics[scale = 0.8]{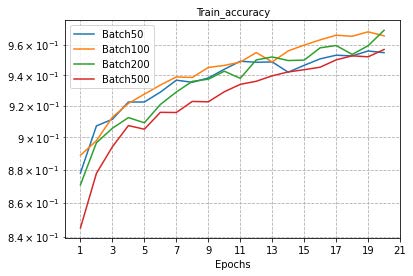} } \vspace{0.5em}
\subfloat[Testing accuracy]{\includegraphics[scale = 0.8]{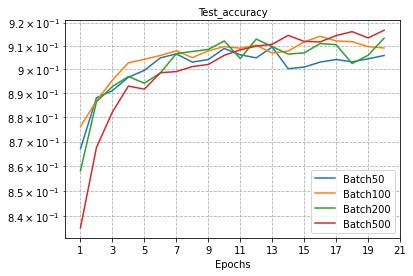} } 
\end{center}
\caption{Performance of SNN: Fashion-MINST datset.}\label{Fashion-MINST}     \vspace{-1em}
\end{figure}
We can see from those figures that our SNN algorithm can generate accurate classification results fairly fast (with various batch sizes) in both training and testing for bench-mark classification tasks. 

\vspace{0.5em}
To present the necessity of applying SNN as a probabilistic learning tool, we consider an adversarial game scenario, in which some ``attacker'' attacks the data and try to fool the neural network (pre-trained on clean dataset) while the model retrains itself along the way using the attacked dataset to play an adversarial game against the attacker. We want to use this experiment to show that our SNN algorithm is more robust than the conventional (deterministic) convolution neural network (CNN) in handling unexpected noises in the data.

\begin{figure}[!htb]
\center
\includegraphics[scale = 0.75]{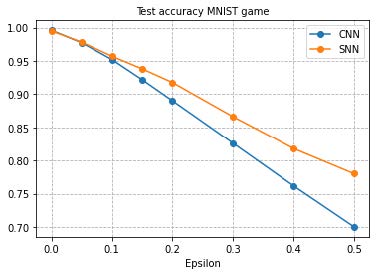} 
\caption{Example 3: Attacked model accuracy with different levels of noise (MNIST).}
    \label{Attack_MNIST}
        \vspace{-1em}
\end{figure}
In Figure \ref{Attack_MNIST}, we compare the accuracy between our SNN with CNN with respect to the level of noises added to the MNIST handwritten dataset, where the same CNN structure is adopted to our SNN model in the drift term. We can see that without noise in the dataset, both SNN and CNN provide very accurate classification results. Then, both SNN and CNN suffer lower accuracy due to the errors added to the dataset by the ``attacker''. However, our SNN algorithm started to outperform the CNN while more and more noises are added. 
\begin{table}\caption{Classification accuracy with different levels of noise: MNIST} \label{Table:MNIST}
\begin{center}
\begin{tabular}{ |c|c|c|c|c|c|c|c|c|} 
\hline
$\epsilon $ & 0.0 & 0.05& 0.1 & 0.15 & 0.2 & 0.3 & 0.4 & 0.5  \\
\hline
Attacked CNN & 0.996& 0.9773& 0.952 & 0.9217& 0.89 & 0.826 & 0.762 & 0.701  \\
\hline
Attacked SNN& 0.995& 0.9782 & 0.957 & 0.938 & 0.917 & 0.866 & 0.818 & 0.770 \\
\hline
\end{tabular}
\end{center}
\end{table}
In Table \ref{Table:MNIST}, we present the comparison of accuracy between SNN and CNN under adversarial attacks. We can see from this table that with $50\%$ of noise added to the data that causes mis-classification, the SNN could still reach $77\%$ of accuracy, which is $7\%$ higher than the conventional CNN method.

In Figure \ref{Attack_Fashion_MNIST}, we compare the accuracy between our SNN with CNN with respect to the level of noises added to the Fashion-MNIST dataset. We can see that without noise in the dataset, both SNN and CNN provide very accurate classification results. Then, both SNN and CNN suffer lower accuracy due to the errors added to the dataset by the ``attacker'', and our SNN algorithm also outperforms CNN while more and more noises are added. 
\begin{figure}[!htb]
\center
\includegraphics[scale = 0.75]{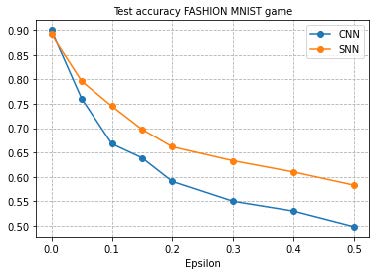} 
\caption{Example 3: Classification accuracy with different levels of noise (Fashion-MNIST).}
    \label{Attack_Fashion_MNIST}
        \vspace{-1em}
\end{figure}
\begin{table}\caption{Classification accuracy with different levels of noise: Fashion-MNIST} \label{Table:Fashion}
\begin{center}\begin{tabular}{ |c|c|c|c|c|c|c|c|c|} 
\hline
$\epsilon $ & 0.0 & 0.05& 0.1 & 0.15 & 0.2 & 0.3 & 0.4 & 0.5  \\
\hline
Attacked CNN & 0.915& 0.760& 0.668 & 0.639& 0.596 & 0.550 & 0.531 & 0.499 \\
\hline
Attacked SNN& 0.910& 0.796 & 0.745 & 0.698 & 0.662 & 0.633 & 0.607 & 0.583\\
\hline
\end{tabular}
\end{center}
\end{table}
In Table \ref{Table:Fashion}, we present the comparison of accuracy between SNN and CNN under adversarial attacks. We can see from this table that with $50\%$ of noise added to the data that causes mis-classification, the SNN could still reach $58\%$ of accuracy, which is $8\%$ higher than the conventional CNN method.



\begin{thebibliography}{10}

\bibitem{agent57-2020}
Adrià~Puigdomènech Badia, Bilal Piot, Steven Kapturowski, Pablo Sprechmann,
  Alex Vitvitskyi, Daniel Guo, and Charles Blundell.
\newblock Agent57: Outperforming the atari human benchmark, 2020.

\bibitem{Bao_SNN}
F.~Bao, Y.~Cao, R.~Archibald, and H.~Zhang.
\newblock A backward sde method for uncertainty quantification in deep
  learning.
\newblock {\em arXiv:2011.14145}, 2021.


\bibitem{Bao_first}
F.~Bao, Y.~Cao, A.~J. Meir, and W.~Zhao.
\newblock A first order scheme for backward doubly stochastic differential
  equations.
\newblock {\em SIAM/ASA J. Uncertain. Quantif.}, 4(1):413--445, 2016.

\bibitem{Bao_Zakai}
F.~Bao, Y.~Cao, C.~Webster, and G.~Zhang.
\newblock A hybrid sparse-grid approach for nonlinear filtering problems based
  on adaptive-domain of the {Z}akai equation approximations.
\newblock {\em SIAM/ASA J. Uncertain. Quantif.}, 2(1):784--804, 2014.

\bibitem{BCZ_2011}
F.~Bao, Y.~Cao, and W.~Zhao.
\newblock Numerical Solutions for Forward Backward Doubly Stochastic
Differential Equations and Zakai Equations.
\newblock {\em Int. J. Uncertain Quantif}, 4(1):351-367, 2011.


\bibitem{BCZ_2015}
F.~Bao, Y.~Cao, and W.~Zhao.
\newblock A First Order Semi-discrete Algorithm for Backward Doubly
Stochastic Differential Equations
\newblock {\em Discrete Contin. Dyn. Syst. Ser. B}, 5(2):1297-1313, 2015.

\bibitem{BCZ_2018}
F.~Bao, Y.~Cao, and W.~Zhao.
\newblock A backward doubly stochastic differential equation approach for
  nonlinear filtering problems.
\newblock {\em Commun. Comput. Phys.}, 23(5):1573--1601, 2018.


\bibitem{BSDE_filter}
F.~Bao and V.~Maroulas.
\newblock Adaptive meshfree backward {SDE} filter.
\newblock {\em SIAM J. Sci. Comput.}, 39(6):A2664--A2683, 2017.



\bibitem{Bottou2018}
L\'{e}on Bottou, Frank~E. Curtis, and Jorge Nocedal.
\newblock Optimization methods for large-scale machine learning.
\newblock {\em SIAM Rev.}, 60(2):223--311, 2018.

\bibitem{Error_ML}
A.~Athalye C.~G.~Northcutt and J.~Mueller.
\newblock Pervasive label errors in test sets destabilize machine learning
  benchmarks.
\newblock {\em arXiv:2103.14749}, 2021.

\bibitem{Chen2018}
Ricky T.~Q. Chen, Yulia Rubanova, Jesse Bettencourt, and David~K Duvenaud.
\newblock Neural ordinary differential equations.
\newblock In S.~Bengio, H.~Wallach, H.~Larochelle, K.~Grauman, N.~Cesa-Bianchi,
  and R.~Garnett, editors, {\em Advances in Neural Information Processing
  Systems}, volume~31, pages 6571--6583. Curran Associates, Inc., 2018.

\bibitem{Dupont2019}
Emilien Dupont, Arnaud Doucet, and Yee~Whye Teh.
\newblock Augmented neural odes.
\newblock In H.~Wallach, H.~Larochelle, A.~Beygelzimer, F.~dAlch\'{e} Buc,
  E.~Fox, and R.~Garnett, editors, {\em Advances in Neural Information
  Processing Systems}, volume~32, pages 3140--3150. Curran Associates, Inc.,
  2019.

\bibitem{E2017}
Weinan E.
\newblock A proposal on machine learning via dynamical systems.
\newblock {\em Commun. Math. Stat.}, 5(1):1--11, 2017.

\bibitem{BSDE_finance}
N.~El~Karoui, S.~Peng, and M.~C. Quenez.
\newblock Backward stochastic differential equations in finance.
\newblock {\em Math. Finance}, 7(1):1--71, 1997.

\bibitem{Gal2016}
Yarin Gal and Zoubin Ghahramani.
\newblock Dropout as a bayesian approximation: Representing model uncertainty
  in deep learning.
\newblock In Maria~Florina Balcan and Kilian~Q. Weinberger, editors, {\em
  Proceedings of The 33rd International Conference on Machine Learning},
  volume~48 of {\em Proceedings of Machine Learning Research}, pages
  1050--1059, New York, New York, USA, 20--22 Jun 2016. PMLR.

\bibitem{Geneva2019}
Nicholas Geneva and Nicholas Zabaras.
\newblock Quantifying model form uncertainty in {R}eynolds-averaged turbulence
  models with {B}ayesian deep neural networks.
\newblock {\em J. Comput. Phys.}, 383:125--147, 2019.

\bibitem{Geneva2020}
Nicholas Geneva and Nicholas Zabaras.
\newblock Modeling the dynamics of {PDE} systems with physics-constrained deep
  auto-regressive networks.
\newblock {\em J. Comput. Phys.}, 403:109056, 32, 2020.

\bibitem{Gerstberger1997}
R.~Gerstberger and P.~Rentrop.
\newblock Feedforward neural nets as discretization schemes for {ODE}s and
  {DAE}s.
\newblock volume~82, pages 117--128. 1997.
\newblock 7th ICCAM 96 Congress (Leuven).

\bibitem{Haber2018}
Eldad Haber and Lars Ruthotto.
\newblock Stable architectures for deep neural networks.
\newblock {\em Inverse Problems}, 34(1):014004, 22, 2018.

\bibitem{ML_future_17}
M.~Miremadi J. Bughin K. George P.~Willmott J.~Manyika, M.~Chui and
  M.~Dewhurst.
\newblock A future that works: Automation, employment, and productivity.
\newblock {\em Technical report, McKinsey Global Institute}, 2017.

\bibitem{Jia2014}
Yangqing Jia, Evan Shelhamer, Jeff Donahue, Sergey Karayev, Jonathan Long, Ross
  Girshick, Sergio Guadarrama, and Trevor Darrell.
\newblock Caffe: Convolutional architecture for fast feature embedding.
\newblock In {\em Proceedings of the ACM International Conference on
  Multimedia}, pages 675--678, New York, NY, USA, 2014. ACM.

\bibitem{Kloeden1992}
Peter~E. Kloeden and Eckhard Platen.
\newblock {\em Numerical solution of stochastic differential equations},
  volume~23 of {\em Applications of Mathematics (New York)}.
\newblock Springer-Verlag, Berlin, 1992.

\bibitem{Kong2020}
Lingkai Kong, Jimeng Sun, and Chao Zhang.
\newblock Sde-net: Equipping deep neural networks with uncertainty estimates,
  2020.

\bibitem{Kwon2020}
Yongchan Kwon, Joong-Ho Won, Beom~Joon Kim, and Myunghee~Cho Paik.
\newblock Uncertainty quantification using {B}ayesian neural networks in
  classification: application to biomedical image segmentation.
\newblock {\em Comput. Statist. Data Anal.}, 142:106816, 17, 2020.

\bibitem{Lak2017}
Balaji Lakshminarayanan, Alexander Pritzel, and Charles Blundell.
\newblock Simple and scalable predictive uncertainty estimation using deep
  ensembles.
\newblock In I.~Guyon, U.~V. Luxburg, S.~Bengio, H.~Wallach, R.~Fergus,
  S.~Vishwanathan, and R.~Garnett, editors, {\em Advances in Neural Information
  Processing Systems}, volume~30, pages 6402--6413. Curran Associates, Inc.,
  2017.

\bibitem{Lecun1998}
Y.~Lecun, L.~Bottou, Y.~Bengio, and P.~Haffner.
\newblock Gradient-based learning applied to document recognition.
\newblock {\em Proceedings of the IEEE}, 86(11):2278--2324, Nov 1998.

\bibitem{Li2021}
Xuechen Li, Ting-Kam~Leonard Wong, Ricky T.~Q. Chen, and David Duvenaud.
\newblock Scalable gradients for stochastic differential equations.
\newblock In Silvia Chiappa and Roberto Calandra, editors, {\em Proceedings of
  the Twenty Third International Conference on Artificial Intelligence and
  Statistics}, volume 108 of {\em Proceedings of Machine Learning Research},
  pages 3870--3882, Online, 26--28 Aug 2020. PMLR.

\bibitem{Liu2020}
Xuanqing Liu, Tesi Xiao, Si~Si, Qin Cao, Sanjiv Kumar, and Cho-Jui Hsieh.
\newblock How does noise help robustness? explanation and exploration under the
  neural sde framework.
\newblock In {\em Proceedings of the IEEE/CVF Conference on Computer Vision and
  Pattern Recognition (CVPR)}, June 2020.

\bibitem{Yong_BSDE}
Jin Ma and Jiongmin Yong.
\newblock {\em Forward-backward stochastic differential equations and their
  applications}, volume 1702 of {\em Lecture Notes in Mathematics}.
\newblock Springer-Verlag, Berlin, 1999.

\bibitem{McDermott2019}
Patrick~L. McDermott and Christopher~K. Wikle.
\newblock Bayesian recurrent neural network models for forecasting and
  quantifying uncertainty in spatial-temporal data.
\newblock {\em Entropy}, 21(2):Paper No. 184, 25, 2019.

\bibitem{Savchenko2020}
Andrey~V. Savchenko.
\newblock Probabilistic neural network with complex exponential activation
  functions in image recognition.
\newblock {\em IEEE Trans. Neural Netw. Learn. Syst.}, 31(2):651--660, 2020.

\bibitem{Sermanet2013}
Pierre Sermanet, David Eigen, Xiang Zhang, Micha{\"{e}}l Mathieu, Rob Fergus,
  and Yann LeCun.
\newblock Overfeat: Integrated recognition, localization and detection using
  convolutional networks.
\newblock {\em CoRR}, abs/1312.6229, 2013.

\bibitem{AlphaGo2017Silver2017}
David Silver, Julian Schrittwieser, Karen Simonyan, Ioannis Antonoglou, Aja
  Huang, Arthur Guez, Thomas Hubert, L~Robert Baker, Matthew Lai, Adrian
  Bolton, Yutian Chen, Timothy~P. Lillicrap, Fan Hui, Laurent Sifre, George
  van~den Driessche, Thore Graepel, and Demis Hassabis.
\newblock Mastering the game of go without human knowledge.
\newblock {\em Nature}, 550:354--359, 2017.

\bibitem{silver2017mastering}
David Silver, Julian Schrittwieser, Karen Simonyan, Ioannis Antonoglou, Aja
  Huang, Arthur Guez, Thomas Hubert, Lucas Baker, Matthew Lai, Adrian Bolton,
  Yutian Chen, Timothy Lillicrap, Fan Hui, Laurent Sifre, George van~den
  Driessche, Thore Graepel, and Demis Hassabis.
\newblock Mastering the game of go without human knowledge.
\newblock 550:354--, October 2017.

\bibitem{EfficientPrO2017Sze}
Vivienne Sze, Yu-Hsin Chen, Tien-Ju Yang, and Joel~S. Emer.
\newblock Efficient processing of deep neural networks: A tutorial and survey.
\newblock {\em Proceedings of the IEEE}, 105:2295--2329, 2017.

\bibitem{td-gammon1994}
G.~{Tesauro}.
\newblock Td-gammon, a self-teaching backgammon program, achieves master-level
  play.
\newblock {\em Neural Computation}, 6(2):215--219, 1994.

\bibitem{Tzen2019}
Belinda Tzen and Maxim Raginsky.
\newblock Neural stochastic differential equations: Deep latent gaussian models
  in the diffusion limit.
\newblock {\em CoRR}, abs/1905.09883, 2019.

\bibitem{Wu2020}
Ling Wu, Kepa Zulueta, Zoltan Major, Aitor Arriaga, and Ludovic Noels.
\newblock Bayesian inference of non-linear multiscale model parameters
  accelerated by a deep neural network.
\newblock {\em Comput. Methods Appl. Mech. Engrg.}, 360:112693, 17, 2020.

\bibitem{DL_Nature_15}
Y.~Bengio Y.~LeCun and G.~Hinton.
\newblock Deep learning.
\newblock {\em Nature}, 512(7553):436--444, 2015.

\bibitem{Yang2021}
Liu Yang, Xuhui Meng, and George~Em Karniadakis.
\newblock B-{PINN}s: {B}ayesian physics-informed neural networks for forward
  and inverse {PDE} problems with noisy data.
\newblock {\em J. Comput. Phys.}, 425:109913, 2021.

\bibitem{Yao2019}
Jiayu Yao, Weiwei Pan, Soumya Ghosh, and Finale Doshi-Velez.
\newblock Quality of uncertainty quantification for bayesian neural network
  inference, 2019.

\bibitem{Yong_control}
Jiongmin Yong and Xun~Yu Zhou.
\newblock {\em Stochastic controls}, volume~43 of {\em Applications of
  Mathematics (New York)}.
\newblock Springer-Verlag, New York, 1999.
\newblock Hamiltonian systems and HJB equations.

\bibitem{Zhang_BSDE}
J.~Zhang.
\newblock {\em Backward Stochastic Differential Equations - From Linear to
  Fully Nonlinear Theory}.
\newblock Probability Theory and Stochastic Modelling. Springer, New York, NY,
  2017.

\bibitem{Zhao_BSDE}
Weidong Zhao, Lifeng Chen, and Shige Peng.
\newblock A new kind of accurate numerical method for backward stochastic
  differential equations.
\newblock {\em SIAM J. Sci. Comput.}, 28(4):1563--1581, 2006.

\end{thebibliography}

\end{document}